\documentclass[11pt,oneside,reqno]{amsart}
\usepackage{amsmath,amsfonts,amssymb}
\usepackage{esint}
\usepackage{amsthm} 
\usepackage{tikz}
\usepackage{dsfont} 

\usepackage[a4paper,left=30mm,right=30mm,top=35mm,bottom=35mm,marginpar=25mm]{geometry} 
\usetikzlibrary{shadings,intersections,patterns}

\usepackage{pgfplots}

\pgfplotsset{compat=1.10}
\usepgfplotslibrary{fillbetween}

\usepackage{xfrac}
\usepackage{xcolor}
\usetikzlibrary{positioning}
\usepackage{mathtools}

\usepackage[shortlabels]{enumitem}

\usepackage{hyperref} 
\usepackage[capitalise, nameinlink, noabbrev]{cleveref} 
\hypersetup{colorlinks=true, 
linkcolor=red,
citecolor=blue,
}

\theoremstyle{plain} 

\newtheorem{theorem}{Theorem}[section]
\newtheorem{proposition}[theorem]{Proposition}
\newtheorem{corollary}[theorem]{Corollary}
\newtheorem{lemma}[theorem]{Lemma}
\newtheorem{definition}[theorem]{Definition}
\theoremstyle{remark}
\newtheorem{remark}[theorem]{Remark}

\newcommand{\R}{\mathbb{R}}

\newcommand{\eps}{\varepsilon}
\newcommand{\de}{\partial}

\newcommand{\HH}{\mathcal{H}}

\DeclareMathOperator{\dist}{dist}

\DeclareMathOperator{\Reg}{Reg}
\DeclareMathOperator{\Sing}{Sing}

\DeclareMathOperator{\Flat}{flat}

\newcommand{\bo}[1]{\boldsymbol{#1}}

 \DeclarePairedDelimiter\abs{\lvert}{\rvert}      
 \DeclarePairedDelimiter\norm{\lVert}{\rVert}    

\makeindex

\def\XXint#1#2#3{{\setbox0=\hbox{$#1{#2#3}{\int}$}
     \vcenter{\hbox{$#2#3$}}\kern-.5\wd0}}

\numberwithin{equation}{section}

\title[Regularity of the two-phase free boundaries]{Regularity of the free boundary for the two-phase Bernoulli problem}

\author[G.~De Philippis]{Guido De Philippis}
\address{\textit{G.~De Philippis:} Courant Institute of Mathematical Sciences, New York University, 251 Mercer St., New York, NY 10012, USA.}
\email{guido@cims.nyu.edu}

\author[L.~Spolaor]{Luca Spolaor}
\address{\textit{L.~Spolaor:} Department of Mathematics, University of California, San Diego, La Jolla, CA, 92093
 }
\email{lspolaor@ucsd.edu}

\author[B.~Velichkov]{Bozhidar Velichkov}
\address{\textit{B.~Velichkov:} Universit\'a degli Studi di Napoli Federico II
Dipartimento di Matematica e Applicazioni ``Renato Caccioppoli'' Via Cintia, Monte S. Angelo I-80126 Napoli, Italy}
\email{bozhidar.velichkov@gmail.com}

\subjclass[2010]{35R35, 49Q10, 47A75}

\keywords{Two-point function, Second variation, Convex potential}\keywords{Free boundary regularity, Two phase Bernoulli problem, Multiphase shape optimization}

\begin{document}

\begin{abstract}
We prove a regularity theorem for the  free boundary of minimizers of the two-phase Bernoulli problem, completing the analysis started by Alt, Caffarelli and Friedman in the 80s. As a consequence, we also show regularity of minimizers of the multiphase spectral optimization problem for the principal eigenvalue of the Dirichlet Laplacian.
\end{abstract}
\maketitle


\section{Introduction}
We consider the two-phase functional $J_{\text{\sc tp}}$ defined, for every open set $D\subset\R^d$ and every function $u:D\to\R$, as
\begin{equation}\label{e:J}\tag{TP}
J_{\text{\sc tp}}(u, D):=\int_{D} |\nabla u|^2\,dx+\lambda^2_+\abs{ \Omega_u^+\cap D}+\lambda^2_-\abs{\Omega_u^-\cap D},
\end{equation}
where the constants $\lambda_+>0$ and $\lambda_->0$ are given and  fixed, and the two phases
$$\Omega_u^+=\{u>0\}\qquad\text{and}\qquad \Omega_u^-=\{u<0\}$$
are the positivity sets of the functions $u^+:=\max\{u,0\}$ and $u^-:=\max\{-u,0\}$. 

\smallskip

We say that a function $u:D\to \R$ is a \emph{local minimizer of $J_{\text{\sc tp}}$ in $D$} if 
$$J_{\text{\sc tp}}\big(u, \Omega\big)\le J_{\text{\sc tp}}\big(v, \Omega\big),$$
for all open sets $\Omega$ and functions \(v:D\to\R\) such that \(\overline \Omega\subset D\) and \(v=u\) on \(D\setminus \Omega\).  

\smallskip
%


In this  paper we aim to study the regularity  of   the free boundary $\partial\Omega_u^+\cup \partial\Omega_u^-\cap D$ for 
local minimizers of $J_{\text{\sc tp}}$ in $D$.
Our main result is a full description of $\partial\Omega_u^+$ and $\partial\Omega_u^-$ around two-phase points: 
\[
\Gamma_{\text{\sc tp}}:=\de \Omega_u^+\cap \de \Omega_u^-\cap D.
\]
 More precisely, we prove that, in a neighborhood of a two-phase point, the sets $\Omega_u^+$ and $\Omega_u^-$ are two $C^{1,\eta}$-regular domains touching along the closed set $\Gamma_{\text{\sc tp}}$.

\begin{theorem}[Regularity around two-phase points]\label{thm:main}
Let $u:D\to\R$ be a local minimizer of $J_{\text{\sc tp}}$ in the open set $D\subset\R^d$. Then, for every two-phase point $x_0\in \Gamma_{\text{\sc tp}}\cap D$, there exists a radius $r_0>0$ (depending on $x_0$) such that $\de \Omega_u^\pm\cap B_{r_0}(x_0)$ are $C^{1,\eta}$ graphs for some \(\eta>0\).
\end{theorem}

Combining \cref{thm:main} with the known regularity  theory  for  {\it one-phase problem},  one obtains the following result, which provides a full description  of the  free boundary of local minimizers of $J_{\text{\sc tp}}$.

\begin{corollary}[Regularity of the  free boundary]\label{cor:full_free_boundary}
Let $u:D\to\R$ be a local minimizer of $J_{\text{\sc tp}}$ in the open set $D\subset\R^d$. Then, each of the  sets $\partial\Omega_u^+\cap D$ and $\partial\Omega_u^-\cap D$ can be decomposed as a disjoint union of a regular and a (possibly empty) singular part 
$$\partial\Omega_u^\pm\cap D=\Reg(\partial\Omega_u^{\pm})\cup \Sing(\partial\Omega_u^{\pm}),$$ 
with the  following properties.
\begin{enumerate}[label=\textup{(\roman*)}]
\item The regular part ${\rm Reg}(\partial\Omega_u^{\pm})$ is a relatively open subset of $\partial\Omega_u^\pm\cap D$ and is locally the graph of a  $C^{1,\eta}$-regular  function, for some $\eta>0$. Moreover, the two-phase free boundary is regular, that is, 
\[
\Gamma_{\text{\sc tp}}\cap D\subset \Reg(\partial\Omega_u^{\pm}).
\]
\item The singular set ${\rm Sing}(\partial\Omega_u^{\pm})$ is a closed subset of $\partial\Omega_u^{\pm}\cap D$ of Hausdorff dimension at most $d - 5$. Precisely, there is a critical dimension\footnote{The critical dimension \(d^*\) is the first dimension, for which there exists a one-homogeneous non-negative local minimizer of the one-phase functional with a singular free boundary. Currently, it is only known that \(5\le d^*\le 7\), \cite{CaJeKe,JeSa,DeJe}.}
 $d^\ast\in[5,7]$ such that 
\begin{itemize}
\item[-] if $d<d^\ast$, then $\Sing(\partial\Omega_u^\pm)=\emptyset$;
\item[-] if $d=d^\ast$, then ${\rm Sing}(\partial\Omega_u^\pm)$ is locally finite in $D$; 
\item[-] if $d> d^\ast$, then  $\Sing(\partial\Omega_u^\pm)$ is a closed \((d-d^\ast)\)-rectifiable  subset of $\partial\Omega_u^\pm\cap D$  with locally finite \(\HH^{d-d^\ast}\) measure.
\end{itemize}	
\end{enumerate}	 
\end{corollary}

As a second corollary of our analysis, by applying the same type of arguments as in \cite{stv} we obtain a complete regularity results for the following shape optimization problem, studied in \cite{bucve, benve, mono}, where the optimal sets have the same qualitative behavior as the sets $\Omega_u^+$ and $\Omega_u^-$ in \cref{cor:full_free_boundary}, contrary to the classical optimal partition problem studied in \cite{caflin1,caflin2,coteve1,coteve2,coteve3} (which corresponds to the case of zero weights $m_i=0$, for every $i$).

\begin{corollary}[Regularity for a multiphase shape optimization problem]\label{t:multiphase}
Let $D$ be a $C^{1,\gamma}$-regular bounded open domain in $\R^d$, for some $\gamma>0$ and $d\ge 2$. Let $n\ge 2$ and  $m_i>0$, $i=1,\dots,n$ be given. Let $(\Omega_1,\dots,\Omega_n)$ be 
a solution of the following optimization problem:
\begin{equation}\label{e:multi}\tag{SOP}
\min\Big\{\sum_{i=1}^n\big(\lambda_1(\Omega_i)+m_i|\Omega_i|\big)\ :\ \Omega_i\subset D\ \text{open}\,;\ \Omega_i\cap\Omega_j=\emptyset\quad\text{for}\quad i\neq j\Big\}. 
\end{equation}
where \(\lambda_1(\Omega_i)\) is the first eigenvalue for the Dirichlet Laplacian in \(\Omega_i\).

Then, the free boundary $\partial\Omega_i$ of each of the sets $\Omega_i$, $i=1,\dots,n$, can be decomposed as the disjoint union of a regular part ${\rm Reg}(\partial\Omega_i)$ and a (possibly empty) singular part ${\rm Sing}(\partial\Omega_i)$, where: 
\begin{enumerate}[label=\textup{(\roman*)}]
\item The regular part ${\rm Reg}(\partial\Omega_i)$ is a relatively open subset of $\partial\Omega_i$ and is locally the graph of a  $C^{1,\eta}$-regular  function, for some $\eta>0$. Moreover, both the contact set with the boundary of the box and the two-phase free boundaries are regular, that is, 
\[
\partial\Omega_i\cap\partial D\subset \Reg(\partial\Omega_i)\quad \text{and}\quad\partial\Omega_i\cap\partial\Omega_j\subset \Reg(\partial\Omega_i)\quad\text{for every}\quad j\in\{1,\dots,n\}\setminus\{i\}.
\]
	\item The singular set $\ {\rm Sing}(\partial\Omega_i)$ is a closed subset of $\partial\Omega_i$ of Hausdorff dimension at most $d - 5$. Precisely, 
	\begin{itemize}
		\item[-] if $d<d^\ast$, then $\Sing(\partial\Omega_i)=\emptyset$,
		\item[-] if $d=d^\ast$, then ${\rm Sing}(\partial\Omega_i)$ is locally finite in $D$,
		\item[-] if $d> d^\ast$, then  $\Sing(\partial\Omega_i)$ is a closed \((d-d^\ast)\)-rectifiable  subset of $\partial\Omega_i$  with locally finite \(\HH^{d-d^\ast}\) measure,
	\end{itemize}	
where $d^\ast\in\{5,6,7\}$  is the critical dimension from \cref{cor:full_free_boundary}.
\end{enumerate}	 
\end{corollary}	 

\subsection{Regularity of  local minimizers of the Bernoulli functional}\label{sub:intro:history}

The study of the regularity of minimizers of  $J_{\text{\sc tp}}$ started  in the seminal paper of Alt and Caffarelli \cite{altcaf}, which was dedicated to the one-phase case, in which  $u$ is non-negative. In this case, it is sufficient to work with the one-phase functional 
\begin{equation}\label{e:J1}\tag{OP}
J_{\text{\sc op}}(u, D):=\int_{D} |\nabla u|^2\,dx+\lambda^2_+\abs{ \Omega_u^+\cap D},
\end{equation}
as the negative phase $\Omega_u^-$ is  empty. In \cite{altcaf} it was proved  that   for a local minimizer $u$  of $J_{\text{\sc op}}$, the free boundary $\partial\Omega_u^+\cap D$ decomposes into a  $C^{1,\eta}$-regular set ${\rm Reg}(\partial\Omega_u^+)$ and a closed singular set ${\rm Sing}(\partial\Omega_u^+)$ of zero $\HH^{d-1}$-Hausdorff measure. A precise estimate on the Hausdorff dimension of ${\rm Sing}(\partial\Omega_u^+)$ was then  given by Weiss in \cite{weiss} as a consequence of his monotonicity formula and  its rectifiability was established by Edelen-Engelstein in \cite{EdEn}. In fact, the results in \cref{cor:full_free_boundary} are an immediate consequence of \cref{thm:main} and the known regularity for  the one-phase parts  
\[
\Gamma^+_{\text{\sc op}}:=\big(\partial\Omega_u^+\setminus\partial\Omega_u^-\big)\cap D\qquad \text{and} \qquad \Gamma^-_{\text{\sc op}}:=\big(\partial\Omega_u^-\setminus\partial\Omega_u^+\big)\cap D.
\]
Indeed:
\begin{itemize}
	\item the regularity of ${\rm Reg}(\partial\Omega_u^\pm)$ (\cref{cor:full_free_boundary} (i)) follows by \cref{thm:main} and  \cite[Theorem 8.1]{altcaf};
	\item the estimates on the dimension of the singular set ${\rm Sing}(\partial\Omega_u^\pm)$ (\cref{cor:full_free_boundary} (ii)) are again a consequence of \cref{thm:main} (which shows that singularities can appear only on the one-phase parts of the free boundary) and the results in  \cite{weiss,EdEn}.
\end{itemize}
 \medskip
 
The regularity of local minimizers with two-phases (that is, local minimizers of $J_{\text{\sc tp}}$ which change sign) was first addressed by Alt, Caffarelli and Friedman in \cite{alcafr}, where 
the authors consider free boundary functionals that weight also the zero level set of $u$:
\begin{equation}\label{e:J3}\tag{\sc ACF}
J_{\text{\sc acf}}(u, D):=\int_{D} |\nabla u|^2\,dx+\lambda^2_+\big|\Omega_u^+\cap D\big|+\lambda^2_-\big|\Omega_u^-\cap D\big|+\lambda^2_0\big|\{u=0\}\cap D\big|,
\end{equation}
where $\lambda_+\ge \lambda_0\ge 0$ and $\lambda_-\ge \lambda_0\ge 0$. When $D\subset \R^2$ is a planar  domain, and under the additional assumptions 
$$\lambda_+\neq\lambda_-\qquad \text{and}\qquad \lambda_0=\lambda_+\, \text{ or }\,\lambda_-,$$  
they showed that the free boundaries  $\partial\Omega_u^+\cap D$ and $\partial\Omega_u^-\cap D$  are $C^1$-regular curves. The key observation here is that the condition 
\begin{equation}\label{e:acf_condition}
\lambda_0=\lambda_+\ \text{ or }\ \lambda_0=\lambda_-\,,
\end{equation} 
forces the level set $\{u=0\}$ to have zero Lebesgue measure. Thus, the two boundaries $\partial\Omega_u^+\cap D$ and $\partial\Omega_u^-\cap D$ coincide and the solution $u$ satisfies the transmission condition  
\begin{equation}\label{e:trans}
|\nabla u^+|^2-|\nabla u^-|^2=\lambda^2_+-\lambda^2_-\qquad\text{on}\qquad \partial\Omega_u^+=\partial\Omega_u^-.
\end{equation}
The free boundary regularity for local minimizers of $J_{\text{\sc acf}}$ in the case \eqref{e:acf_condition} is already known in any dimension. Indeed, 
the regularity of the free boundary $\partial\Omega_u^+=\partial\Omega_u^-$, for functions which are harmonic (or solve an elliptic PDE) in $\Omega_u^+\cup\Omega_u^-$ and satisfy the transmission condition \eqref{e:trans}, is today well-understood, after the seminal work of Caffarelli  \cite{caf1,caf2,caf3} (see also  the book  \cite{cs}) and the more recent results of De Silva-Ferrari-Salsa  \cite{dfs, dfs2, dfs3}, which are based on the techniques introduced by De Silva in \cite{desilva} and which are central also in the present paper.

To the best of our knowledge, the only known regularity result  for minimizers of \eqref{e:J3} in the case when 
\begin{equation}\label{e:no_acf_condition}
\lambda_+>\lambda_0\ \text{ and }\ \lambda_->\lambda_0,
\end{equation} 
is due to the second and  third authors in  \cite{spve}, where it is  proved that, in dimension \(d=2\),  the free boundaries $\partial\Omega_u^+$ and $\partial \Omega_u^-$ are \(C^{1,\eta}\) regular. The proof relies on a novel  epiperimetric type inequality which applies only in dimension two and it  was recently extended (still in dimension two) to almost-minimizers by the same two authors and Trey in \cite{stv}.




\medskip

In this paper, we complete the analysis started by Alt, Caffarelli and Friedman in \cite{alcafr}, by proving a regularity result for the free boundaries of the local minimizers of \eqref{e:J3}, in the case \eqref{e:no_acf_condition} and  in any dimension $d\ge 2$.
Indeed, \cref{thm:main} and \cref{cor:full_free_boundary} apply directly to \eqref{e:J3} as 
the local minimizers of  \eqref{e:J3}, corresponding to the parameters $\lambda_0$, $\lambda_+$ and $\lambda_-$, are local minimizers of \eqref{e:J} with parameters 
\[
\lambda_+'=\sqrt{\lambda_+^2-\lambda_0^2}\quad \text{ and }\quad \lambda_-'=\sqrt{\lambda_-^2-\lambda_0^2}\,.
\]

%

\subsection{One-phase, two-phase and branching points on the free boundary}\label{sub:intro:branching}$ $\\
Let $u:B_1\to\R$ be a (local) minimizer of $J_{\text{\sc tp}}$ in $B_1$ and let, as above, $\Omega_u^\pm=\{\pm u>0\}$. Notice that, the zero level set $\{u=0\}$ might have positive Lebesgue measure in $B_1$ and also non-empty interior, contrary to what happens with the minimizers of \eqref{e:J3} with $\lambda_+=\lambda_0$. This introduces a new element in the analysis of the free boundary, which can now switch from one-phase to two-phase at the so-called branching points, at which the zero level set looks like a cusp. 
Precisely, this means that the free boundary $\partial\Omega_u^+\cap B_1$ (the same holds for the negative phase $\partial\Omega_u^-\cap B_1$) can be decomposed into:
\begin{itemize} 
\item a set of one-phase points $\Gamma_{\text{\sc op}}^+:=\partial\Omega_u^+\setminus \partial\Omega_u^-\cap B_1$, and 
\item a set of two-phase points $\Gamma_{\text{\sc tp}}:=\partial\Omega_u^+\cap \partial\Omega_u^-\cap B_1$. 
\end{itemize}

By definition the set of one-phase points $\Gamma_{\text{\sc op}}^+$ is relatively open in $\partial\Omega_u^+$. Precisely, if $x_0\in \Gamma_{\text{\sc op}}^+$, then there is a ball $B_r(x_0)$ which does not contain points from the negative phase, $B_r(x_0)\cap\Omega_u^-=\emptyset$. Thus, $u$ is a minimizer of the one-phase functional $J_{\text{\sc op}}$ in $B_r(x_0)$ and the regularity of $\partial\Omega_u^+\cap B_r(x_0)$ follows from the results in \cite{altcaf,weiss}. 

For what concerns the two-phase points, we can  further  divide them into {\it interior} and {\it branching} points:
\begin{itemize} 
	\item we say that $x_0$ is an {\it interior} two-phase point, $x_0\in \Gamma_{\text{\sc tp}}^{\rm int}$, if $x_0\in \Gamma_{\text{\sc tp}}$ and 
	$$\big|B_r(x_0)\cap \{u=0\}\big|=0\quad\text{for some}\quad r>0\,;$$ 
	\item conversely, we say that $x_0$ is a {\it branching} point, $x_0\in \Gamma_{\text{\sc tp}}^{\rm br}$, if $x_0\in \Gamma_{\text{\sc tp}}$ and 
	$$\big|B_r(x_0)\cap \{u=0\}\big|>0\quad\text{for every}\quad r>0\,.$$ 
\end{itemize}
By definition, $\Gamma_{\text{\sc tp}}^{\rm  int}$ is an open subset of $\partial\Omega_u^+\cap B_1$. In particular, $u$ is a minimizer of the Alt-Caffarelli-Friedman functional \eqref{e:J3} with $\lambda_+=\lambda_0$ in a small ball $B_r(x_0)$ and the regularity of $\Gamma_{\text{\sc tp}}^{\text{\rm int}}$ is a consequence of the results in \cite{alcafr,caf1,caf2,caf3, cs, dfs, dfs2,dfs3}. 

In order to complete the study of the regularity of the free boundaries one has then to focus on the branching points. Note that by the previous discussion \(|\nabla u^+|\) is a H\"older continuous function on \(\Gamma_{\text{\sc op}}^+\cup\Gamma_{\text{\sc tp}}^{\rm int}\). By relying on the results of \cite{desilva}, to prove \cref{thm:main} one has to show that $|\nabla u^+|:\partial\Omega_u^+\cap B_1\to\R$ is H\"older continuous across the branching points
$$
\Gamma_{\text{\sc tp}}^{\text{\rm br}}=\left(\partial\Omega_u^+\cap B_1\right)\setminus \left(\Gamma_{\text{\sc op}}^+\cup\Gamma_{\text{\sc tp}}^{\rm int}\right).
$$ 
By following \cite{spve} and \cite{stv} this will be consequence of 
\begin{center}
\emph{ uniform ``flatness'' decay at the two-phase points $x_0\in \Gamma_{\text{\sc tp}}$,}
\end{center}	
which is the main result of our paper.

\subsection{Flatness decay at the two-phase points}\label{sub:intro:flatness}
By the Weiss' monotonicity formula (see \cite{weiss}), at every two-phase point \(x_0\in \Gamma_{\text{\sc tp}}\), the limits of  blow-up sequences
\[
u_{x_0,r_k}(y)=\frac{u(x_0+r_k y)}{r_k}
\]
are {\it two-plane solutions} of the form 
\begin{equation}\label{e:blowup}\tag{\textsc{TpS}}
\begin{aligned}
H_{\alpha,\bo{e}}(x)&=\alpha (x\cdot \bo{e})^+-\beta(x\cdot \bo{e})^-,
\\
&\text{with}\qquad \bo{e} \in \mathbb S^{d-1},\quad \alpha^2-\beta^2=\lambda_+^2-\lambda_-^2,\quad\text{and}\quad \alpha\ge \lambda_+,\quad\beta\ge \lambda_-.
\end{aligned}
\end{equation}
However, \emph{a priori} the limiting profile might depend on the chosen sequence. As it is usual in this type of problems, uniqueness of the blow-up profile (and thus regularity of \(u\)) is a consequence of a uniform flatness (or excess) decay.
 
Given \(u\), its flatness  in $B_r(x_0)$ with respect to $H=H_{\alpha,\bo{e}}$ is defined as 
$$\Flat_{B_r(x_0)}\,(u,H) =\frac1r\|u-H\|_{L^\infty(B_r(x_0))},$$
In particular, we can assume that the flatness  becomes small at a uniform scale in a neighborhood of any $x_0\in \Gamma_{\text{\sc tp}}$. Precisely, for every $\eps>0$ and  $x_0\in \Gamma_{\text{\sc tp}}$, there is $r>0$ and a neighborhood $\,\mathcal U$ of $x_0$, such that  
$$\Flat_{B_r(y_0)}\,(u,H) \le \eps\quad\text{for every}\quad y_0\in \mathcal U\cap \Gamma_{\text{\sc tp}}.$$
Our aim is to prove that there is a universal threshold $\eps >0$ such that
$$\Flat_{B_r(x_0)}\,(u,H) \le \eps\quad\text{for some two-plane solution}\quad H=H_{\alpha,\bo{e}},$$
then it improves in the ball $B_{\sfrac{r}2}(x_0)$, which means that there exists another two-plane solution $\widetilde H=H_{\tilde{\alpha},\tilde{\bo{e}}}$ such that
\begin{equation}\label{e:excess_decay_intro}
\Flat_{B_{\sfrac{r}2}(x_0)}\,(u,\widetilde H)\le 2^{-\gamma}\,\Flat_{B_r(x_0)}\,(u,H),
\end{equation}
for some small, but universal, \(\gamma>0\). 

In order to prove \eqref{e:excess_decay_intro},
 we argue by contradiction. That is, there is a sequence of  minimizers $u_k$ and a sequence of two-plane solutions \(H_k\), such that 
\[
\eps_k:=\norm{u_k-H_k}_{L^\infty(B_1)}\to 0 \qquad\text{but}\qquad\inf_{\widetilde H}\,\norm{u_k-\widetilde H}_{L^\infty(B_{\sfrac{1}{2}})}\ge 2^{-\gamma}\eps_k\,,
\]  
where the infimum is taken over all $\widetilde H$ of the form \eqref{e:blowup}.

\smallskip

Now, the two key points of the argument are to show that the sequence 
\begin{equation*}
v_k:=\frac{u_k-H_k}{\eps_k}
\end{equation*}
is (pre-)compact in a suitable topology and that any limit point \(v_\infty\) is a solution of a suitable ``linearized'' problem (that turns out to be a non-linear one); then the regularity theory for the limiting problem allows to obtain the desired contradiction. 
Let us briefly analyze these two main steps of the proof.

\smallskip

\noindent {\it The ``linearized'' problem}. The nature of the limiting problem depends on the type of free boundary point one is considering. At branching points (the ones that we are most interested in),  \(v_\infty\) turns out be the the solution of a  \emph{two-membrane problem}, \eqref{e:twomembrane}. 
 At interior two-phase points $\Gamma_{\text{\sc tp}}^{\rm int}$, we instead  recover a transmission problem as in \cite{dfs}.
Note that in the first case, the ``linearized'' problem is actually non-linear. Similar phenomena have been already observed in a number of related situation: in this same context, a derivation of the limiting problem was done in \cite{ASW},  while for  Bernoulli type problems a similar fact appears in studying regularity close to the boundary of the container, \cite{changsavin}. See also  \cite{FiSe,savinyu} for similar issues  in studying  the singular set of obstacle type problems. Heuristically linearizing to an ``obstacle'' type problem is due to the fact that there is a natural ``ordering'' between the negative and the positive phases of any possible competitor. Note instead if one linearizes the plain one phase problem, the natural linearized problem is the Neumann one, this was observed in \cite{AW} (in the parabolic case) and fully exploited in \cite{desilva}, see also \cite{CG,CGS} where other non-local type problems appear as linearization.

\smallskip

\noindent {\it Compactness of the linearizing sequence $v_k$}. 
We follow the approach introduced by  De~Silva in  \cite{desilva}, which is based on a  partial Harnack type inequality, introduced in different context by Savin in  \cite{savin1,savin2}. This  is a weaker form of the flatness decay estimate \eqref{e:excess_decay_intro} that does not take into account the scaling of the functional (which means that it cannot be used to obtain the regularity of the free boundary in a direct way). The rough idea is that if $\|u-H\|_{L^\infty(B_{r}(x_0))}$ falls below a certain (universal) threshold, then $u$ is closer to $H$ in the ball $B_{\sfrac{r}2}(x_0)$, precisely:
\begin{equation}\label{e:excess_decay_intro-2}
\|u-H\|_{L^\infty(B_{\sfrac{r}2}(x_0))}\le 2^{-\delta }\|u-H\|_{L^\infty(B_{r}(x_0))},
\end{equation}
for some $\delta>0$. This estimate implies the 
compactness  of the sequence $v_k$ by a classical (Ascoli-Arzel\`a type) argument. 

For local minimizers of the one-phase functional \eqref{e:J1} or the two-phase functional \eqref{e:J3} with coefficients satisfying the condition \eqref{e:acf_condition}, the functions $H$ can be chosen in the respective class of blow-up limits. In fact, for the one-phase problem, it is sufficient to take $H$ to be the (possibly translated and rotated) one-homogeneous global one-phase solution $H(x)=\lambda_+x_d^+$ (as in \cite{desilva}); for the two-phase problem in the case \eqref{e:acf_condition}, it is sufficient to take $H$ in the class of two-plane solutions \eqref{e:blowup}, precisely as in \cite{dfs}. However, in our case, it turns out that the class of two-plane solutions is not large enough. The reason  is that there  exist solutions which are arbitrarily  close to a two-plane solution of the form  \(H_{\lambda_+,\bo{e}_d}\) but which are not a smooth perturbation of it. 
For instance  the function, 
\begin{equation}\label{e:misonotagliatounditocazzo}
H(x)=\lambda_+(x_d+\eps_1)^+-\lambda_-(x_d-\eps_2)^-,
\end{equation}
 is \(\max\{\eps_1, \eps_2\}\)-close to the two-plane solution \(H_{\lambda_+,\bo{e}_d}\), but  \eqref{e:excess_decay_intro-2} fails for it.

This is not just a technical difficulty. In fact, in order to get the compactness of the linearizing sequence, the partial improvement of flatness \eqref{e:excess_decay_intro-2} is not needed just at one two-phase point $x_0$, but in all the points in a neighborhood of $x_0$.  
Now, since at a branching point, the behavior of the free boundary switches from two-phase (which roughly speaking corresponds to the case when the two free boundaries $\partial\Omega_u^+$ and  $\partial\Omega_u^-$ coincide) to one-phase (in which the two free boundaries $\partial\Omega_u^+$ and  $\partial\Omega_u^-$ are close to each other but separate, as on {\sc Figure 1} below), the class of reference functions $H$ has to contain both the two-plane solutions \eqref{e:blowup} and the solutions of the form \eqref{e:misonotagliatounditocazzo}.

\begin{center}
\includegraphics[scale=1]{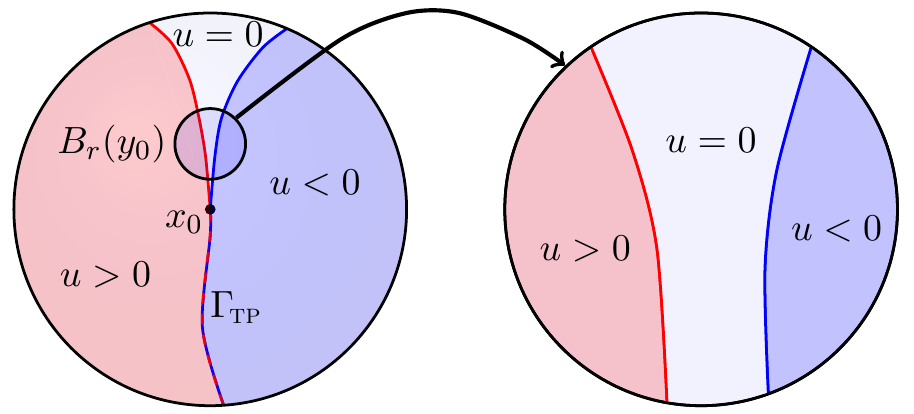}

{\sc Figure 1.} In a neighborhood of a branching point $x_0$ the solution\\ $u$ might get closer to the solution of the form \eqref{e:misonotagliatounditocazzo}.
\end{center}

\medskip

\subsection*{Structure of the paper} This paper is organized as follows: in \cref{sec:2} we recall some basic properties of minimizers  and we fix the notation; in \cref{sec:3} we  establish the excess decay lemma; in \cref{sec:4} we prove our main results; in \cref{s:app,app:twomembrane} we collect the proofs of some technical facts.

\subsection*{Acknowledgements} L.S. was partially supported by NSF grant DMS-1951070. B.V. has been partially supported by Agence Nationale de la Recherche (ANR) with the projects GeoSpec (LabEx PERSYVAL-Lab, ANR-11-LABX-0025-01) and CoMeDiC (ANR-15-CE40-0006). 

\medskip

At the final stage of the preparation of this work, the authors have been informed that two other groups are working on similar problems, namely in \cite{ASW} the authors aim to establish a result analogous to the ours via variational techniques, while in \cite{DETV} the goal is to prove the same result for almost minimizers in the spirit of \cite{dt,det,detv}.

\section{Basic properties of minimizers}\label{sec:2}

In this section we recall (mostly without proof) some basic properties of local minimizers of $J_{\text{\sc tp}}$. In particular, in \cref{sub:reg_min} we recall  Lipschitz-regularity and  non-degeneracy property  of  $u$;  \cref{sub:class_blow-up} is dedicated to the study of blow-up limits of $u$ at two-phase points and  in \cref{sub:viscosity} we show that $u$ satisfies an optimality condition in viscosity sense.

\subsection{Regularity of  minimizers}\label{sub:reg_min} Let $u$ be a local minimizer of $J_{\text{\sc tp}}$. Then, it is well-known that $u$ is locally Lipschitz continuous and  non-degenerate. 

Throughout this paper, we will assume that the weights in \eqref{e:J} are ordered as follows: 
\begin{equation}\label{e:ordine}
\lambda_+\ge \lambda_->0.
\end{equation}
Notice that this is not restrictive as one can always replace \(u\) by \(-u\) in $J_{\text{\sc tp}}$.

\begin{proposition}[Lipschitz regularity and non-degeneracy of local minimizers]\label{l:properties_of_min}
	Let $D\subset\R^d$ be an open set, $\lambda_+\geq \lambda_->0$, and $u$ be a local minimizer of $J_{\text{\sc tp}}$. Then the following properties hold:
	\begin{enumerate}[label=\textup{(\roman*)}]
		\item  {\rm Lipschitz continuity.} $u\in C^{0,1}_{\mathrm{loc}}(D)$.
		\item {\rm Non-degeneracy.} There is  constant $\alpha=\alpha(d, \lambda_\pm)>0$ such that 
		$$\displaystyle{\fint_{\de B_r(x_0)}u^\pm\geq \alpha\,r}\quad\text{for every $x_0\in \overline{\Omega_u^\pm}\cap D$ and every $0<r<\dist(x_0,\partial D)$}.$$
	\end{enumerate}
\end{proposition}

\begin{proof}
The second claim was first proved in  \cite[Theorem 3.1]{alcafr} and depends only on the fact that each of the two phases $\Omega_u^+$ and $\Omega_u^-$ is optimal with respect to one-sided inwards perturbations (see for instance \cite{bucve} and \cite[Section 4]{notes}). The Lipschitz continuity of $u$ is more involved and requires the use of the Alt-Caffarelli-Friedman monotonicity formula and the non-degeneracy of $u^+$ and $u^-$. It was first proved in \cite[Theorem 5.3]{alcafr}, see also the recent paper \cite{detv} for quasi-minimizers. 
\end{proof}

\subsection{Blow-up sequences and blow-up limits}\label{sub:class_blow-up}
Let $u$ be a local minimizer of $J_{\text{\sc tp}}$ in the open set $D\subset\R^d$. For every $x_0\in \partial\Omega_u\cap D$  and every$0<r<\dist(x_0,\de D)$, we consider the function 
$$u_{x_0,r}(x):=\frac{u(x_0+rx)}{r},$$
which is well-defined for $|x|<\frac1r \dist(x_0,\de D)$ and vanishes at the origin. Give  a sequence    $r_k>0$ such that 
$r_k \to 0$, we say that the sequence of functions $u_{x_0,r_k}$ is a \emph{blow-up sequence}. Note that, for every $R>0$, and $k\gg 1$, the functions $u_{x_0,r_k}$ are defined on the ball $B_R$, vanish at  zero and are uniformly Lipschitz  in $B_R$. Hence, there is a Lipschitz continuous function $v:\R^d\to\R$ and a (no relabeled) subsequence of $u_{x_0,r_k}$  such that $u_{x_0,r_k}$ converges to $v$ uniformly on every ball $B_R\subset \R^d$. We say that $v$ is a blow-up limit of $u$ at $x_0$. Notice that $v$ might depend not only on $x_0$ and $u$ but also on the (sub-)sequence $r_k$. We will denote by $\mathcal{BU}(x_0)$  the collection of all possible blow-up limits of $u$ at $x_0$.

The following lemma  classifies all the possible elements of  $\mathcal{BU}(x_0)$ when  $x_0 \in \Gamma_{\text{\sc tp}}$.The result is well-known and we only sketch the proof for the sake of completeness.
\begin{lemma}[Classification of the blow-up limits]\label{l:class_blow-up_limits}
Let $u$ be a local minimizer of $J_{\text{\sc tp}}$ in the open set $D\subset\R^d$, and let $v$ be a blow-up limit of $u$ at the two-phase point $x_0\in \Gamma_{\text{\sc tp}}$. 
Then, $v$ is of the form 
		\begin{equation*}\label{e:class_blow-up_form}
		v(x)=H_{\alpha,\bo{e}}(x)=\alpha (x\cdot \bo{e})^+-\beta(x\cdot \bo{e})^-, 
		\end{equation*}
		where \(\bo{e} \in \mathbb S^{d-1}\), and $\alpha,\beta$ are such that
		\begin{equation*}\label{e:class_blow-up_constants}
		\ \alpha^2-\beta^2=\lambda_+^2-\lambda_-^2\qquad\text{and}\qquad \alpha\ge \lambda_+,\quad\beta\ge \lambda_-.
		\end{equation*}
\end{lemma}
\begin{proof}
Let $v$ be a blow-up limit of $u$ at $x_0$ and let $u_{x_0,r_k}$ be a blow-up sequence converging to $v$ (locally uniformly in $\R^d$).  First, notice that the non-degeneracy of $u$, \cref{l:properties_of_min} (ii), implies that $v$ is non trivial and changes sign: $v^+\not\equiv 0$ and $v^-\not\equiv 0$. Moreover, since every $u_{x_0,r_k}$ is a local minimizer of $J_{\text{\sc tp}}$ ( it is standard to infer that $v$ is also a local minimizer of $J_{\text{\sc tp}}$ in $\R^d$ (see for instance \cite[Section 6]{notes}). Thus, $v$ is harmonic on $\Omega_v^+$ and $\Omega_v^-$.  On the other hand, by the Weiss monotonicity formula, \cite{weiss}, $v$ is one-homogeneous, in polar coordinates:
\[
v(\rho, \theta)=\rho V(\theta)
\] 
In particular  $V$ is an eigenfunction of the spherical Laplacian $\Delta_{\mathbb S}$ on the spherical sets $\Omega_v^\pm\cap \mathbb S^{d-1}$:
\begin{equation}\label{e:spherical_equation}
-\Delta_{\mathbb S}V^\pm=(d-1)V^\pm\quad\text{in}\quad \Omega_v^\pm\cap \mathbb S^{d-1}.
\end{equation}
We now choose  $c>0$ such that 
\[
 \int_{\mathbb S^{d-1}}(V^+-cV^-)d\HH^{d-1}=0.
 \]
  Using the \eqref{e:spherical_equation} and integrating by parts, we get that 
$$
\int_{\mathbb S^{d-1}}|\nabla_\theta (V^+-cV^-)|^2\,d\HH^{d-1}=(d-1)\int_{\partial B_1}|V^+-cV^-|^2\,d\HH^{d-1},
$$
This means that $V^+-cV^-$ is an eigenfunction of the spherical Laplacian on  $\mathbb S^{d-1}$, corresponding to the eigenvalue $(d-1)$. Since  the $(d-1)$-eigenspace contains only linear functions one easily deduce that  $v$ is of the form \eqref{e:blowup}. 

Conditions \eqref{e:class_blow-up_constants} can be obtained by a smooth variation of the free boundary $\{v=0\}$. Indeed, if considering  competitors of the form $v_t(x)=v(x+t\xi(x))$ for smooth  compactly vector fields $\xi$, and taking the derivative of $J_{\text{\sc tp}}(v_t,B_1)$ at $t\to0$, we get that 
$$\int_{\{v=0\}\cap B_1}(\bo{e}\cdot \xi)\Big(|\nabla v^+|^2-|\nabla v^-|^2-\big(\lambda_+^2-\lambda_-^2\big)\Big)\,d\HH^{d-1}=0,$$
which by the arbitrariness of \(\xi\)  is precisely the first part of \eqref{e:class_blow-up_constants}. The second part of \eqref{e:class_blow-up_constants} is analogous and follows by considering competitors of the form $v_t(x)=v^+(x)-v^-(x+t\xi(x))$ for vector fields  with $\xi\cdot \bo{e}\le 0$  so that it  moves  negative phase only inwards, that is, $\{v_t<0\}\subset \{v<0\}$. Taking the derivative of the energy at $t>0$, we get 
$$\int_{\{v=0\}\cap B_1}(\xi\cdot \bo{e})\big(|\nabla v^-|^2-\lambda_-^2\big)\,d\HH^{d-1}\le 0,$$ 
which gives $\beta\ge\lambda_-$. The estimate on $\alpha$ is analogous.  
\end{proof}	

We record the following consequence of  \cref{l:class_blow-up_limits} which says that the ``flatness'' can be chosen uniformly small in a neighborhood of a two-phase point. 

\begin{corollary}\label{c:flatness}
Let $u$ be a local minimizer of $J_{\text{\sc tp}}$ in the open set $D\subset\R^d$, and let $x_0$ be a two-phase point $x_0\in \Gamma_{\text{\sc tp}}$. Then, for every $\eps>0$ there are $r>0$ and $\rho>0$, and a function $H_{\alpha, \bo{e}}$ of the form \eqref{e:blowup} such that: 
$$\|u_{y_0,r}-H_{\alpha, \bo{e}}\|_{L^\infty(B_1)}\le \eps\quad\text{for every}\quad y_0\in B_\rho(x_0).$$ 
\end{corollary}
	
\begin{proof}
By \cref{l:class_blow-up_limits}, there exists $r>0$ and $H$ such that $\|u_{x_0,r}-H\|_{L^\infty(B_1)}\le \sfrac{\eps}2$. On the other hand, by the Lipschitz continuity of $u$ 
$$\|u_{x_0,r}-u_{y_0,r}\|_{L^\infty(B_1)}\le \frac{L}{r}|x_0-y_0|.$$ 
Choosing $\rho$ small enough (such that $\frac{L\rho}{r}\le \sfrac\eps2$), we get the claim.
\end{proof}

\subsection{Optimality conditions at the free boundary}\label{sub:viscosity}
Let $u:D\to \R$ be a local minimizer of $J_{\text{\sc tp}}$. In this section, we will show that $u$ satisfies the following optimality conditions at two-phase free boundary points:
\begin{equation}\label{e:tp_viscosity}
 |\nabla u^+|^2-|\nabla u^-|^2=\lambda_+^2-\lambda_-^2\qquad\text{and}\qquad  |\nabla u^\pm|\ge \lambda_\pm \qquad\text{on}\qquad  \Gamma_{\text{\sc tp}}.
\end{equation}
We notice that if $u$ was differentiable at $x_0\in\Gamma_{\text{\sc tp}}$, that is,
\begin{equation}\label{e:diff_tp}
\begin{cases}
u^+(x)=(x-x_0)\cdot\nabla u^+(x_0)+o(|x-x_0|)\quad\text{for every}\quad x\in\Omega_u^+,\\
u^-(x)=(x-x_0)\cdot\nabla u^-(x_0)+o(|x-x_0|)\quad\text{for every}\quad x\in\Omega_u^-,
\end{cases}
\end{equation}
then \eqref{e:tp_viscosity} would be an immediate consequence of \cref{l:class_blow-up_limits}. Of course,  differentiability of $u^+$ and $u^-$ (and the uniqueness of the blow-up limits\footnote{It is immediate to check that, if the blow-up is unique at $x_0\in\Gamma_{\text{\sc tp}}$, that is, 
$$\displaystyle\lim_{r\to0}\|u_{x_0,r}-H\|_{L^\infty(B_1)}=0\qquad\text{for some $H$ as in \eqref{e:blowup}},$$ 
then $\alpha=|\nabla u^+|(x_0)$, $\beta=|\nabla u^-|(x_0)$ and \eqref{e:diff_tp} does hold.}) is not a priori known, so we will use the optimality condition in some weak (viscosity) sense, based on comparison with (more regular) test functions. 

\begin{definition} Let $D$ be an open set.
	\begin{enumerate}[label=\textup{(\roman*)}]
		\item  We say that a function $Q:D\to\R$ \emph{touches} a function \(w:D\to\R\) from below (resp. from above) at a point \(x_0\in D\) if $Q(x_0)=w(x_0)$ and  
		$$Q(x)-w(x)\le 0 \quad \bigl( \text{resp. }\, Q(x)-w(x)\ge 0 \bigr),$$
		for every $x$ in a neighborhood of $x_0$. We will say that $Q$ touches $w$ strictly from below (resp. above), if the above inequalities are strict for \(x\neq x_0\). 
		
		\item A function $Q$ is an admissible \emph{comparison function} in $D$ if 
		\begin{enumerate}[label=\textup{(\alph*)}]
			\item $Q\in C^1\big(\overline{\{Q>0\}}\cap D\big)\cap C^1\big(\overline{\{Q<0\}}\cap D\big)$;
			\item $Q\in C^2\big(\{Q>0\}\cap D\big)\cap C^2\big(\{Q<0\}\cap D\big)$;
			\item $\partial \{Q>0\}$ and $\partial \{Q<0\}$ are smooth manifolds in $D$. 
		\end{enumerate}
	\end{enumerate}
\end{definition}

The optimality conditions on $u$ are given in the next lemma. Before we give the precise statement, we recall that $\partial\Omega_u\cap D=\Gamma_{\text{\sc op}}^+\cup \Gamma_{\text{\sc op}}^-\cup \Gamma_{\text{\sc tp}},$ where 
$$\Gamma_{\text{\sc op}}^+:=\partial\Omega_u^+\setminus\partial\Omega_u^-\cap D,\quad \Gamma_{\text{\sc op}}^-:=\partial\Omega_u^-\setminus\partial\Omega_u^+\cap D\quad\text{and}\quad \Gamma_{\text{\sc tp}}:=\partial\Omega_u^-\cap\partial\Omega_u^+\cap D.$$
\begin{lemma}[The local minimizers are viscosity solutions]\label{l:optimality}
Let $u$ be a local minimizer of $J_{\text{\sc tp}}$ in the open set $D\subset \R^d$. Then, $u$ in harmonic in $\Omega_u^+\cup\Omega_u^-$ and satisfies the following optimality conditions on the free boundary $\partial\Omega_u\cap D$.
\begin{enumerate}[label=\textup{(\Alph*)}]
\item Suppose that $Q$ is a comparison function that touches $u$ from \underline{below}  at $x_0$.
\begin{enumerate}[label=\textup{(A.\arabic*)}]
\item If $x_0\in \Gamma_{\text{\sc op}}^+$, then $\abs{\nabla Q^+(x_0)} \le \lambda_+ $;
\item if $x_0\in \Gamma_{\text{\sc op}}^-$, then $Q^+\equiv 0$ in a neighborhood of $x_0$ and $\abs{\nabla Q^-(x_0)} \ge \lambda_-$;
\item if $x_0\in \Gamma_{\text{\sc tp}}$, then $\abs{\nabla Q^-(x_0)} \ge \lambda_- $ and 
$$\abs{\nabla Q^+(x_0)}^2-\abs{\nabla Q^-(x_0)}^2\le \lambda_+^2-\lambda_-^2.$$
\end{enumerate}
\item Suppose that $Q$ is a comparison function that touches $u$ from \underline{above}  at $x_0$.
\begin{enumerate}[label=\textup{(B.\arabic*)}]
	\item If $x_0\in \Gamma_{\text{\sc op}}^+$, then  $Q^-\equiv 0$ in a neighborhood of $x_0$ and $\abs{\nabla Q^+(x_0)} \ge \lambda_+$;
	\item if $x_0\in \Gamma_{\text{\sc op}}^-$, then $\abs{\nabla Q^-(x_0)} \le \lambda_-$;
	\item if $x_0\in \Gamma_{\text{\sc tp}}$, then $\abs{\nabla Q^+(x_0)} \ge \lambda_+$ and 
	$$\abs{\nabla Q^+(x_0)}^2-\abs{\nabla Q^-(x_0)}^2\ge \lambda_+^2-\lambda_-^2.$$
\end{enumerate}

\end{enumerate}
\end{lemma}
\begin{proof}
If $x_0$ is a one-phase point, then the gradient bounds in (A.1), (A.2), (B.1) and (B.2) follow by \cite[Proposition 7.1]{notes}, the claims $Q^+\equiv0$ in (A.2) and $Q^-\equiv 0$ in (B.2) being trivially true. Suppose now that $x_0\in \Gamma_{\text{\sc tp}}$ and that $Q$ touches $u$ from below at $x_0$. Let $u_{x_0,r_k}$ and $Q_{x_0,r_k}$ be blow-up sequences of $u$ and $Q$ at $x_0$. Then, up to extracting a subsequence, we can assume that  $u_{x_0,r_k}$ converges uniformly to a blow-up limit $H_u\in \mathcal {BU}(x_0)$ of the form 
$$H_u(x)=\alpha(x\cdot \bo{e})_+-\beta (x\cdot \bo{e})_-.$$
On the other hand, since $Q^+$ and $Q^-$ are differentiable at $x_0$ (respectively in $\overline\Omega_Q^+$ and $\overline\Omega_Q^-$), we get that 
$Q_{x_0,r_k}$ converges to the function
$$H_Q(x)=|\nabla Q^+(x_0)|(x\cdot \bo{e}')_+-|\nabla Q^-(x_0)| (x\cdot \bo{e}')_-,$$
where $\bo{e}'=|\nabla Q^+(x_0)|^{-1}\nabla Q^+(x_0)=-|\nabla Q^-(x_0)|^{-1}\nabla Q^-(x_0)$. Now since, $H_Q$ touches $H_u$ from below (and since $\alpha\neq0$ and $\beta\neq0$), we have that $\bo{e}'=\bo{e}$ and 
$$\abs{\nabla Q^+(x_0)}^2-\abs{\nabla Q^-(x_0)}^2\le\alpha^2-\beta^2\qquad\text{and}\qquad \abs{\nabla Q^+(x_0)} \le \alpha,\quad\abs{\nabla Q^-(x_0)} \ge \beta.$$
Combined with \eqref{e:class_blow-up_constants}, this gives (A.3). The proof of (B.3) is analogous.
\end{proof}

In particular, if $u:D\to\R$ is a continuous function such that the claims (A) and (B) hold for every comparison function $Q$,
 then we say that $u$ satisfies the following overdetermined condition on the free boundary in viscosity sense: 
 \begin{equation}\label{e:viscosity_solution_def}
 \begin{cases}
 |\nabla u^+|^2-|\nabla u^-|^2=\lambda_+^2-\lambda_-^2 ,\,
|\nabla u^+|\ge \lambda_+\,\text{and}\, |\nabla u^-|\ge \lambda_- &\text{on}\quad \Omega_u^+\cap\Omega_u^-\cap D\,;\\
  |\nabla u^+|=\lambda_+&\text{on}\quad  D\cap \Omega_u^+\setminus\Omega_u^-\,; \\
   |\nabla u^-|=\lambda_-&\text{on}\quad  D\cap \Omega_u^-\setminus\Omega_u^+\,.
   
   \end{cases}
 \end{equation}
 Thus, \cref{l:optimality} can be restated as follows: {\it If $u$ is a local minimizer of $J_{\text{\sc tp}}$ in $D$, then it satisfies \eqref{e:viscosity_solution_def} in viscosity sense.}
 
We conclude this section by recording the following  straightforward  consequence of definition of viscosity solution, where we consider what happens when a function is touching only  one of the two phases (note that in the second item we are restricting the touching points only to the one-phase free boundaries)

\begin{lemma}\label{l:NY}
Let \(u: D \to \R\) be a continuous  function which satisfies \eqref{e:viscosity_solution_def}. 

\begin{enumerate}[label=\textup{(\roman*)}]

\item Assume that \(Q\) is a comparison function touching \(u^+\) from above at \(x_0\in \partial \Omega_{u}^+\) (resp. \(-u^-\) from below at \(x_0\in \partial \Omega_{u}^-\)) , then 
\[
|\nabla Q^+|(x_0)\ge \lambda_+ \qquad \Bigl(\text{resp.\quad }   |\nabla Q^-|(x_0)\ge \lambda_-\Bigr).
\] 

\item Assume that \(Q\) is a comparison function touching \(u^+\) from below  at \(x_0\in \Gamma_{\textsc{op}}^+\) (resp. \(-u^-\) from above  at \(x_0\in \Gamma_{\textsc{op}}^-\)) , then 
\[
|\nabla Q^+|(x_0)\le \lambda_+ \qquad \Bigl(\text{resp.\quad}   |\nabla Q^-|(x_0)\le \lambda_-\Bigr).
\] 
\end{enumerate}
\end{lemma}

\begin{proof}
The claim (i) simply  follows by, for instance,  noticing that the assumption implies that \(Q\ge u^+\ge 0\) so that \(Q\) touching \(u\) from above and thus one can apply \(B.1\) and the first part of \(B.3\) in the definition of viscosity solution  and that a symmetric argument holds for \(u^{-}\). 

Concerning claim (ii), we note that  since \(x_0\in \Gamma_{\textsc{op}}^+\), \(u\ge 0\) in a neighborhood  of \(x_0\). In particular, the function \(Q^+\) is touching \(u\) from below at \(x_0\) and thus the conclusion follows by (B.2) in the definition of viscosity solution. 
\end{proof}

\section{Flatness decay}\label{sec:3}

In this section we prove that, at two-phase points, the flatness decays from one scale to the next. Our main result is the following theorem, which applies to any viscosity solution of the two-phase problem.

\begin{theorem}[Flatness decay for viscosity solutions]\label{t:eps_regularity}
For every $L\ge \lambda_+\ge\lambda_->0$ and $\gamma\in(0,\sfrac12)$, there exist $\eps_0>0$, $C>0$ and $\rho\in(0,\sfrac14)$ such that the following holds. 
Suppose that the function $u:B_1\to\R$ satisfies: 
\begin{enumerate}[label=\textup{(\alph*)}]
\item $u$ is \(L\)-Lipschitz continuous; 
\item zero is on the two-phase free boundary,  \(0\in \Gamma_{\text{\sc tp}}=\partial \Omega_{u}^{+}\cap\partial \Omega_{u}^{-}\);
\item $u$ is harmonic in $\Omega_{u}^{+}\cup \Omega_{u}^{-}$;
\item $u$ satisfies the optimality condition \eqref{e:viscosity_solution_def} in viscosity sense;
\item $u$ is $\eps_0$-flat in $B_1$, that is, 
\begin{equation}\label{e:start}
\|u-H_{\alpha,\bo{e}_d}\|_{L^\infty(B_1)}\le \eps_0\qquad\text{for some}\qquad L\ge \alpha\ge\lambda_+\,.
\end{equation}
\end{enumerate}
Then, there are $\bo{e}\in\mathbb S^{d-1}$ and $\tilde\alpha\ge \lambda_+$ such that 
\begin{equation}\label{e:iof_normal_oscillation}
|\bo{e}-\bo{e}_d|+|\tilde\alpha-\alpha|\le C\,\|u-H_{\alpha,\bo{e}_d}\|_{L^\infty(B_1)},
\end{equation}
and
\begin{equation}\label{e:iof_thesis}
\|u_\rho-H_{\tilde\alpha,\bo{e}}\|_{L^\infty(B_1)}\le \rho^{\gamma}\,\|u-H_{\alpha,\bo{e}_d}\|_{L^\infty(B_1)}.
\end{equation}
\end{theorem}

The  proof of \cref{t:eps_regularity} follows easily combining the two upcoming lemmas. In the first one we deal with the situation where the two-plane solution is, roughly, \(H_{\lambda_+}\). Note that this is the situation where one might expect the presence of branching points and it is indeed in this setting that we will obtain the two membrane problem as ``linearization''. In the second lemma, we deal with the case when the closest half-plane solution has a gradient much larger than $\lambda_+$. We will later show that in this case the origin is an interior two-phase point. 

\begin{lemma}[Improvement of flatness: branching points]\label{l:small_regime}
For every $L\ge \lambda_+\ge\lambda_->0$,  $\gamma\in(0,\sfrac12)$,  and \(M>0\), there exist $\eps_1=\eps_1(\gamma,d,L, M)$, \(C_1=C_1(\gamma,d,L, M)\) and $\rho=\rho(\gamma,d,L,M)$ such that the following holds. For every function $u:B_1\to\R$ satisfying (a)-(b)-(c)-(d) of \cref{t:eps_regularity} and such that
\begin{equation*}
\norm{u-H_{\alpha,\bo{e}_d}}_{L^\infty(B_1)}\le \eps_1 \,,\quad\text{with}\quad  0\le \alpha-\lambda_+\le  M\norm{u-H_{\alpha,\bo{e}_d}}_{L^\infty(B_1)},
\end{equation*}
there exist $\bo{e}\in\mathbb S^{d-1}$ and $\tilde\alpha\ge\lambda_+$, for which \eqref{e:iof_normal_oscillation} and \eqref{e:iof_thesis} hold.  
\end{lemma}

\begin{lemma}[Improvement of flatness: non-branching points]\label{l:big_regime}
For every $L\ge \lambda_+\ge\lambda_->0$ and $\gamma\in(0,1)$, there exist $\eps_2=\eps_2(\gamma,d, L)$, \(\overline{M}=\overline{M}(\gamma,d,L)\) and $\rho=\rho(\gamma,d,L)$ \(C_2=C_2(\gamma,d, L)\) such that the following holds. For every function $u:B_1\to\R$ satisfying (a)-(b)-(c)-(d) of \cref{t:eps_regularity} and such that
\begin{equation*}
\norm{u-H_{\alpha,\bo{e}_d}}_{L^\infty(B_1)}\le \eps_2 \,,\quad\text{with}\quad \alpha-\lambda_+\ge \overline{M}\norm{u-H_{\alpha,\bo{e}_d}}_{L^\infty(B_1)} ,
\end{equation*}
there exist $\bo{e}\in\mathbb S^{d-1}$ and $\,\tilde\alpha\ge\lambda_+$, for which \eqref{e:iof_normal_oscillation} and \eqref{e:iof_thesis} hold.   
\end{lemma}

Let us first show that  \cref{t:eps_regularity} follows from \cref{l:small_regime} and \cref{l:big_regime}.

\begin{proof}[\bf Proof of \cref{t:eps_regularity}]
Fix \(\gamma \in (0,\sfrac{1}{2})\) and notice that \(\alpha <2L\), where \(L\) is the Lipschitz constant of \(u\).  Next choose  $M=2\overline{M}$ in \cref{l:small_regime}, where $\overline{M}$ is as in \cref{l:big_regime}. Let $\eps_0=\min\big\{\eps_2(2\overline{M}),\sfrac{\eps_1}{2}\big\}$. Then, we can apply either \cref{l:small_regime} or \cref{l:big_regime}.
\end{proof}

In order to prove \cref{l:small_regime} and \cref{l:big_regime}, we will argue by contradiction. Hence in the following we consider a sequence \(u_k\)  of minimizers 
 such that 
\begin{equation}\label{e:ualphak}
\eps_k:=\norm{u_k-H_{\alpha_k,\bo{e}_d}}_{L^\infty(B_1)}\to 0\qquad\text{and}\qquad \lambda_+\le \alpha_k\le L,
\end{equation}
where 
$$ \|\nabla u_k\|_{L^\infty(B_1)}\leq L\qquad\text{for every}\qquad k\ge 1.$$
We also set 
\begin{equation}\label{e:ell}
\ell:= \lambda_{+}^2\lim_{k\to\infty}\frac{\alpha^2_k-\lambda^2_+}{2\alpha^2_k\eps_k}=\lambda_{-}^2\lim_{k\to\infty}\frac{\beta^2_k-\lambda^2_-}{2\beta^2_k\eps_k}
\end{equation}
which  we can assume to exists up to extracting a subsequence. It might be useful to keep in mind that   \(\ell=\infty\) will correspond to \cref{l:big_regime} while  \(0\leq\ell\leq M<\infty\)   to \cref{l:small_regime}.

\medskip
 
In order to prove \cref{l:big_regime} and \cref{l:small_regime}, we will first show that the sequence 
\begin{equation}\label{e:vk}
v_k(x)=
\begin{cases}
v_{+,k}(x):=\dfrac{u_k(x)-\alpha_k x_d^+}{\alpha_k\eps_k}\qquad&x\in \Omega_{u_k}^+\cap B_1
\\
v_{-,k}(x):=\dfrac{u_k(x)+\beta_k x_d^-}{\beta_k\eps_k}\qquad&x\in \Omega_{u_k}^-\cap B_1
\end{cases}
\end{equation}
is compact in some suitable sense; we give the precise statement in  \cref{cor:comp} below and we postpone the proof to \cref{sub:compactness}. We then establish in \cref{l:limitprob} the limiting problem solved by  its limit \(v\). Note that this problem  depends on the value of $\ell$ which is distinguishing  whether we are or not at branching points.

Finally, in \cref{sub:proof_iof} we show how to deduce \cref{l:big_regime} and \cref{l:small_regime} from \cref{cor:comp} and \cref{l:limitprob}. For the remainder of the paper we will denote with 
$$
B_r^{\pm}:=B_r\cap \{ x_d^\pm>0 \}\,,\qquad \text{for every }r>0\,.
$$

\begin{lemma}[Compactness of the linearizing sequence  $v_k$]\label{cor:comp}
	Let  $u_k$ be a sequence of functions satisfying (a), (b), (c) and (d) of \cref{t:eps_regularity} uniformly in \(k\) and let  \(\eps_k\)  and \(\alpha_k\) be as in \eqref{e:ualphak} and let \(v_k\) be defined by \eqref{e:vk}. Then there are H\"older continuous functions 
	\[
	v_+ : \overline{B^+_{\sfrac12}}\to\R\qquad\text{and}\qquad  v_-: \overline{B^-_{\sfrac12}} \to\R,
	\]
	with 
	\[
	v_+\le v_-\text{ on } B_{\sfrac12}\cap \{x_d=0\}, \qquad  v_+ (0)=v_-(0)=0,
	\]
	and such that the sequences of closed graphs
	\[
	\Gamma_k^\pm:=\Big\{(x, v_{\pm,k}(x))\ :\ x\in \overline{\Omega_{u_k}^\pm\cap B_{\sfrac12}}\Big\},
	\]
	converge, up to a (non-relabeled) subsequence,  in the Hausdorff distance to the closed graphs
	\[
	\Gamma_\pm=\Big\{(x, v_\pm(x))\ :\ x\in \overline{B_{\sfrac12}^\pm}\Big\}.
	\]
	In particular, the following claims hold. 
	\begin{enumerate}[label=\textup{(\roman*)}]
		\item For every $\delta>0$, $v_{\pm,k}$ converges uniformly to $v\pm$ on $B_{\sfrac12}\cap\{ \pm x_d>\delta\}$. 
		\item For every sequence $x_k\in \overline{\Omega_{u_k}^\pm}\cap B_1$ converging to $x\in \overline{B_{\sfrac12}^\pm}$, we have 
		$$ v_\pm(x)=\lim_{k\to\infty} v_{\pm,k}(x_k).$$
		\item For every $x\in\{x_d=0\}\cap B_{\sfrac12}$ , we have
		$$v_\pm(x)=\mp\lim_{k\to\infty}\frac{x_k\cdot \bo{e}_d}{\alpha_k\eps_k}\quad\text{for any sequence}\quad \partial\Omega_{u_k}^\pm\ni x_k\to x.$$
	\end{enumerate}
	In particular,  $\{x_d=0\}\cap \overline B_{\sfrac12}$ decomposes into a open  jump set 
	\[
	\mathcal J=\{v_+<v_-\}\cap \{x_d=0\}\cap \overline B_{\sfrac12},
	\]
	and its complementary contact set
	\[
	\mathcal C=\{v_+=v_-\}\cap \{x_d=0\}\cap \overline B_{\sfrac12}.
	\]
Furthermore, if \(x\in \mathcal J\), then
\begin{equation}\label{e:nodp}
\liminf_{k\to\infty}\dist\big(x,\partial \Omega_{u_k}^+\cap\partial\Omega_{u_k}^-\big)>0.
\end{equation}
In particular for all \(x\in \mathcal J\), there exists two sequences \(x^\pm_k\in \Gamma_{k,\textsc{op}}^{\pm}\) such that \(x_k^\pm\to x\).
\end{lemma}

In the next lemma  we determine the limiting problem solved by the function \(v\) defined as
\begin{equation}
\label{e:defv}
v(x)=
\begin{cases}
v_+(x)\quad\text{for}\quad x\in B^{+}_{\sfrac{1}{2}},
\\
v_-(x)\quad\text{for}\quad x\in B^{-}_{\sfrac{1}{2}},
\end{cases}
\end{equation}
where  $v_+$ and $v_-$ are as in \cref{cor:comp}.

\begin{lemma}[The ``linearized'' problem]\label{l:limitprob}
	Let  $u_k$, \(\eps_k\)  and \(\alpha_k\) be as in \eqref{e:ualphak},  \(v_k\) be defined by \eqref{e:vk} and \(\ell\) as in \eqref{e:ell}.  Let also  \(v_\pm\) be as in  \cref{cor:comp}:
	
\medskip
\noindent 
 $\bo{\ell=\infty:}$ Then \(\mathcal J=\emptyset\) and \(v_\pm\) are viscosity solutions of the transmission problem:
		\begin{equation}\label{e:trasmission}
		\begin{cases}
		\Delta v_{\pm}=0 \qquad&\text{in \(B^\pm_{\sfrac{1}{2}}\)}
		\\
		\alpha_\infty^2\partial_{d} v_{+}  = \beta_\infty^2\partial_{d} v_{-}  &\text{on \(B^\pm_{\sfrac{1}{2}}\cap \{x_d=0\}\)}
		\end{cases}
		\end{equation}
		where \(\alpha_\infty=\lim_k \alpha_k\) and \(\beta_\infty=\lim_k \beta_k\), which we can assume to exist up to extracting a further subsequence.

\medskip
\noindent 
$\bo{0\leq\ell<\infty:}$ Then \(v\) is a viscosity solution of the two membrane problem:
		\begin{equation}\label{e:twomembrane}
		\begin{cases}
		\Delta v_{\pm}=0 \qquad&\text{in \(B^\pm_{\sfrac{1}{2}}\)}
		\\
		\lambda^2_{\pm} \partial_{d} v_{\pm}+\ell\ge 0 &\text{in \(B_{\sfrac{1}{2}}\cap\{x_d=0\}\)}
		\\
		\lambda_{\pm}^2\partial_{d} v_{\pm}  +\ell=0 &\text{in \(\mathcal J\)}
		\\
		\lambda_+^2\partial_{d} v_{+} = \lambda_-^2\partial_{d} v_{-}  &\text{in \(\mathcal C\)}
		\\
		v_+\leq v_- & \text{in \(B_{\sfrac{1}{2}}\cap \{x_d=0\}\) }
		\end{cases}\,.
		\end{equation}
\end{lemma}

\begin{remark}\label{def:viscosity_linearized}
	Here by \emph{viscosity solution of \eqref{e:trasmission} and \eqref{e:twomembrane}} we mean a  function $v$ as in \eqref{e:defv} such that \(v_\pm\) are continuous in  \(\overline{B^\pm_{\sfrac{1}{2}}}\), $\Delta v_\pm=0$ in $B^\pm_{\sfrac{1}{2}}$ and such that the following holds. 
	
\begin{itemize}
\item[-]	If we are in case \eqref{e:trasmission}, let  \(p,q\in \R\) and let \(\tilde P\) be a smooth function such that \(\partial_d \tilde P=0\). Suppose that \(\tilde{P}\) is subharmonic (superharmonic) and that the function
	\[
	P:=px_{d}^{+}-qx_{d}^{-}+\tilde{P}
	\]
	touches \(v\) strictly from below (above) at a point \(x_0\in B_{\sfrac12}\cap \{x_d=0\}\), then 
	\[
	\alpha^2_\infty p\le \beta^2_\infty q\qquad \Bigl( \alpha^2_\infty p\ge \beta^2_\infty q\Bigr)\,.
	\]
\item[-]	If we are in case \eqref{e:twomembrane} then
	\begin{enumerate}
		\item if $P_\pm$ is a smooth superharmonic function in $B_{\sfrac12}^\pm$ touching $v_\pm$ strictly from above at $x_0\in B_{\sfrac12}\cap \{x_d=0\}$, then $\lambda_{\pm}^2\de_d P_\pm \ge 0$;
		\item if $P_\pm$ is a smooth subharmonic function in $B_{\sfrac12}^\pm$ touching $v_\pm$ strictly from below at $x_0\in \mathcal J$, then $\lambda_{\pm}^2\de_d P_\pm \le 0$; 
		\item if \(p,q\in \R\) and \(\tilde P\) is a smooth subharmonic (superharmonic) function such that \(\partial_d \tilde P=0\) and such that the function
		\[
		P:=px_{d}^{+}-qx_{d}^{-}+\tilde{P}
		\]
		touches \(v\) strictly from below (above) at a point \(x_0\in B_{\sfrac12}\cap \{x_d=0\}\), then 
		\[
		\lambda_+^2p \le \lambda_-^2 q \qquad \Bigl(\lambda_+^2p \ge \lambda_-^2q \Bigr)\,.
		\]
	\end{enumerate}
	
	\end{itemize}

\end{remark}

\subsection{Compactness of the linearizing sequence. Proof of \cref{cor:comp}}\label{sub:compactness}

\noindent The key point in establishing a suitable compactness for  \(v_k\) is a  ``partial Harnack'' inequality, in the spirit of \cite{desilva,dfs}. As explained in the introduction, in dealing with branching points one needs to work separately on the positive and negative part. An additional difficulties arise also at pure two-phase points since we want also to deal with the case  \(\lambda_-=\lambda_+\). Let us briefly explain the ideas of the proof.

If \(u\) is close in \(B_1\) to  a global solution of the form \(H_{\alpha,\bo{e_d}}\) with \(\alpha >\lambda_+\), then we expect that in a small neighborhood $B_\rho$ of the origin the level set $\{u=0\}$ has zero Lebesgue measure and that all the free boundary points in $B_\rho$ are ``interior" two-phase points (indeed, at the end, this will be a consequence of  the $C^{1}$ regularity of $u$ and of the free boundary). In this case one expects to be able to do the same argument as in \cite{dfs}. This is true except for the  following caveat, if one wants to deal with the case \(\lambda_-=\lambda_+\) then the sliding arguments used in \cite{desilva,dfs} (see also  \cite{caf1,caf2}) does not yield the desired contradiction since the positive term might actually be zero. For this reason one has first to ``increase'' the slope of the trapping solution, so that the sliding argument would give the desired contradiction. Namely if \(u\) is trapped between two translation of a two-plane solution:
\[
H_{\alpha, \bo{e}_d}(x+b)\le u\le H_{\alpha, \bo{e}_d}(x+a)
\]
in say \(B_1\) and at the point \(P=(0,\dots,0,\sfrac{1}{2})\) \(u\) is closer to \(H_{\alpha, \bo{e}_d}(\cdot+a)\) then to \(H_{\alpha, \bo{e}_d}(\cdot+b)\), we can increase in a quantitative way the slope of the positive part of the lower two-plane solution in half ball, i.e. 
\[
 u\ge \alpha'(x+b)^+-\beta(x+b)^+ \qquad \alpha'>\alpha,
\]
see \cref{lm:improvement}. The sliding argument of \cite{desilva,dfs} then allows to translate this to a a (quantitative) increase of \(b\), yielding the  partial decay of flatness of the free boundary. This is the situation studied in \cref{lm:partialharnack1}.

If instead \(u\) is close to \(H_{\lambda_+,\mathbf{e}_d}\) then the free boundary can behave in several different ways. 
Indeed, in this case  the origin  can be either an interior two-phase point, a branching two-phase point but it might also happen that
$$
u(x)\approx \lambda_+ (x_d+\eps_1)_+-\lambda_- (x_d-\eps_2)_-\quad\text{with}\quad 0<\eps_1,\eps_2\ll1.
$$
Since  as explained in the introduction we have to deal with all the of the above situations  we have to prove a decay in this situation is  to improve separately the positive and the negative parts of $u$. More precisely if in \(B_1\) 
\[
\lambda_{+}\bigl(x_d+ b_+\bigr)^+\le u^+(x)\le  \lambda_{+} \bigl(x_d+ a_+\bigr)^+.
\qquad 
-\lambda_{-}\bigl(x_d+ b_-\bigr)^-\le -u^-(x)\le - \lambda_{-}\bigl(x_d+ a_-\bigr)^-
\]
for suitable \(a_{\pm}, b_{\pm}\), one wants to  find new constants $\bar a_{\pm}, \bar b_{\pm}\in$ with 
\[
(\bar b_--\bar a_-)<(b_--a_-)\qquad (\bar b_+-\bar a_+)<(b_+-a_+)
\]
and for which, in half the ball, 
\[
\lambda_{+}\bigl(x_d+ \bar b_+\bigr)^+\le u^+(x)\le  \lambda_{+} \bigl(x_d+ \bar a_+\bigr)^+.
\qquad 
-\lambda_{-}\bigl(x_d+ \bar b_-\bigr)^-\le -u^-(x)\le - \lambda_{-}\bigl(x_d+ \bar a_-\bigr)^-.
\]
Here one has to distinguishes the case in which, say, the lower function
\[
\lambda_{+}\bigl(x_d+ \bar b_+\bigr)^+-\lambda_{-}\bigl(x_d+ \bar b_-\bigr)^-
\]
is looks like a two plane solution, i.e \(b_+-b_-\ll1\), or not and to perform different comparisons according to the situation. This dealt in \cref{lm:partialharnack2}.

We start with the following simple lemma which allows to ``increase'' the slop of the comparison functions.

\begin{lemma}\label{lm:improvement}
There is a dimensional constants \(\tau=\tau(d)>0\) such that the following hold.   Assume that   \(v: B_1\to \R\)  is  a continuous function with \(\Delta v=0\) on \(\{v>0\}\) and such that
\[
\lambda \bigl(x_d+b\bigr)^+\le v \le  \lambda \bigl(x_d+a\bigr)^+\,,
\]
for some \(a,b\in \big(-\sfrac{1}{100},  \sfrac{1}{100}\big)\).  Let  \(P=( 0, \dots, 0, \sfrac{1}{2})\), then for all \(\eps\in (0,\frac{1}{2})\)
\begin{align*}
v(P)&\le \lambda(1-\eps)\Bigl(\frac{1}{2}+a\Bigr)^+ && \Longrightarrow  &&v\le \lambda(1-\tau \eps) \bigl(x_d+a\bigr)^+\quad \text{in}\quad B_{\sfrac{1}{4}}(0)\,,
\end{align*}
and
\begin{align*}
v(P)&\ge \lambda(1+\eps) \Bigl(\frac{1}{2}+b\Bigr)^+&& \Longrightarrow  &&  v\ge \lambda(1+\tau \eps) \bigl(x_d+b\bigr)^+ \quad \text{in}\quad  B_{\sfrac{1}{4}}(0).
\end{align*}
\end{lemma}

\begin{proof}
We  prove only the first implication, since the second one can be obtained by the same arguments.  First, we notice that, since \(b \le \sfrac{1}{100}\),  both \(v\) and  \(\lambda(x_d+a)_+\) are positive and  harmonic in \(B_{\sfrac{1}{4}}(P)\). Thus, 
\[
\lambda(x_d+a)^+-v\ge 0\quad \text{in} \quad B_{\sfrac{1}{4}}(P)
\]
and 
\[
\lambda\Bigl(\frac{1}{2}+a\Bigr)^+-v(P)\ge \lambda \eps\Bigl(\frac{1}{2}+a\Bigr)^+\ge \frac{49}{100} \lambda \eps
\]
Hence, by Harnack inequality and the bound \(\abs{a}\le \sfrac{1}{100}\) there are dimensional constants \(\bar c\) and \(c\) such that 
\[
v(x) \le \lambda \bigl(x_d+a\bigr)^+-\lambda \bar c  \eps \le  \lambda (1-c\eps)\bigl(x_d+a\bigr)^+ \qquad \text{for all \( x \in  B_{\sfrac{1}{8}}(P)\).}
\]
 We now let \(w\) be the solution of the following problem
\[
\begin{cases}
\Delta w=0\qquad &\text{in \(B_{1}(0)\setminus B_{\sfrac{1}{8}}(P)\cap \{x_d > -a\}\)}
\\
w=0 &\text{on \(B_{1}\cap \{x_d = -a\}\)}
\\
w= \lambda \bigl(x_d+a\bigr)^+ &\text{on \(\partial B_{1}(0)\cap \{x_d > -a\}\)}
\\
w=\lambda (1-c\eps\bigr)\bigl(x_d+a\bigr)^+ &\text{on \(\partial B_{\sfrac{1}{8}}(P)\cap \{x_d > -a\}\).}
\end{cases}
\]
By the Hopf Boundary Lemma, 
\[
 w(x)\le (1-\tau\eps)(x_d+a)^+ \quad \text{for every $x$ in}\quad  B_{\sfrac{1}{4}}\cap \{x_d > -a\},
\]
for a suitable  constant \(\tau=\tau(d)\). Since, by the comparison principle, \(u\le w\), this concludes the proof.
\end{proof}

We next prove the two partial Harnack inequalities:

We distinguish two cases:

\smallskip

The proof of the partial Harnack inequality is based on comparison with suitable test functions. In order to build these ``barriers", we will often use the following function $\varphi$. Let \(Q=( 0, \dots, 0, \sfrac{1}{5})\) and we let \(\varphi: B_1\to R\) be defined by:
\begin{equation}\label{e:phi}
\varphi (x)=
\begin{cases}1&\text{ if \(x\in B_{\sfrac{1}{100}}(Q),\)}\\
\kappa_d \Big(|x-Q|^{-d}-\big(\sfrac34\big)^{-d}\Big) 
& \text{ if \(x\in  B_{\sfrac{3}{4}}(Q)\setminus \overline B_{\sfrac{1}{100}}(Q)\),}\\
0 &\text{ otherwise},
\end{cases}
\end{equation}
where the dimensional constant $\kappa_d$ is chosen in such a way that  \(\varphi \) is continuous. 

It is immediate to check that \(\varphi\) has the following properties:
\begin{enumerate}[label=($\varphi$.\arabic*)]
\item \label{i:P1} \(0\le \varphi\le 1\) in $\R^d$, and \(\varphi=0\) on \(\partial B_1\);
\item \label{i:P2} \(\Delta \varphi \ge c_d>0\) in  \(\{\varphi>0\}\setminus \overline B_{\sfrac{1}{100}}(Q)\);
\item \label{i:P3} \(\partial_d \varphi  >0 \) in \(\{\varphi>0\}\cap \{\abs{x_d}\le \sfrac{1}{100}\}\);
\item \label{i:P4} \( \varphi  \ge c_d>0\) in \(B_{\sfrac{1}{6}}\).
\end{enumerate}
where \(c_d\) is a dimensional constant.

\begin{lemma}[Partial Boundary Harnack I]\label{lm:partialharnack2} Given \(\lambda_+\ge \lambda_->0\) there exist constants $\bar\eps=\bar \eps(d, \lambda_{\pm})>0$ and $\bar{c}=\bar{c}(d, \lambda_{\pm})\in(0,1)$  such that, for every function  $u:B_4\to \R$ satisfying (a), (c) and (d) in \cref{t:eps_regularity},  the following property holds true.

Let   $a_{\pm}, b_{\pm}\in \bigl(-\sfrac{1}{100},\sfrac{1}{100}\bigr)$ be such that
\[
b_-\le a_-, \qquad b_+\le a_+,\qquad  b_-\le b_+,\qquad a_-\le a_+,
\]
and
\[
(a_--b_-)+( a_+-b_+)\le \bar \eps.
\]
Assume that for  \(x \in B_4\):
\[
\lambda_{+}\bigl(x_d+ b_+\bigr)^+\le u^+(x)\le  \lambda_{+} \bigl(x_d+ a_+\bigr)^+.
\]
and
\[
-\lambda_{-}\bigl(x_d+ b_-\bigr)^-\le -u^-(x)\le - \lambda_{-}\bigl(x_d+ a_-\bigr)^-
\]
Then, one can find new constants $\bar a_{\pm}, \bar b_{\pm}\in \bigl(-\sfrac{1}{100},\sfrac{1}{100}\bigr)$, with 
\[
\bar b_-\le \bar  a_-, \qquad \bar  b_+\le \bar  a_+,\qquad  \bar  b_-\le \bar  b_+,\qquad \bar  a_-\le \bar  a_+,
\]
and
\[
\bar  a_--\bar  b_-\le \bar c (a_--b_-)  \qquad \bar  a_+-\bar  b_+\le  \bar c (a_+-b_+)
\]
such that for \(x \in B_{\sfrac{1}{6}}\):
\[
\lambda_{+}\bigl(x_d+\bar b_+\bigr)^+\le u^+(x)\le  \lambda_{+} \bigl(x_d+\bar a_+\bigr)^+.
\]
and
\[
-\lambda_{-}\bigl(x_d+\bar b_-\bigr)^-\le -u^-(x)\le - \lambda_{-}\bigl(x_d+\bar a_-\bigr)^-
\]
\end{lemma}

\begin{proof}
Let us show how to improve the positive part. More precisely we show how given \(a_+, a_-, b_+, b_-\) as in the statement we can find \(\bar a_+\) and \(\bar b_+\). The proof for \(\bar b_-\) and \(bar a_-\) works in the same way and is left to the reader.

We let 
\[
P=( 0, \dots, 0,  2)
\]
 and we distinguish two  cases:

\medskip
\noindent
\(\bullet\) \emph{Case 1. Improvement from above.} Assume that, at the point $P$, $u^+$ is closer to  $\lambda_+(2+b_+)^+$ than to the upper barrier $\lambda_+(2+a_+)^+$. Precisely  that 
\[
u^+(P)\le \lambda_+(2+a_+)^+-\frac{\lambda_+ (a_+-b_+)}{2}.
\]
In this case, we will show that \(u\) is below  \(\lambda_+(x+\bar a_+)^{+}\) in a smaller ball centered at the origin  for \(\bar a_+\) strictly smaller than \(a_+\). 


We start by setting    
\[
\eps:=a_+-b_+ \le \bar \eps.
\]
Then
\[
u^+(P)\le\lambda_+(2+a_+)^+-\frac{\lambda \eps}{2} \le  \lambda_+(1- c \eps) (2+a_+)^+
\]
for a suitable (universal)  constant \(c\). We can thus apply  (the scaled version of)  \cref{lm:improvement} to \(u^+\), to infer the existence of  a dimensional  constant \(\tau\) such that 
\begin{equation}\label{e:inizio}
u^+\le \lambda_+(1-\tau \eps)\bigl(x_d+a_+\bigr)^+\quad \text{in}\quad B_{1}.
\end{equation}
For \(\varphi\) as in \eqref{e:phi} and \(t\in[0,1]\) we set 
\[
f_t=\lambda_+(1-\tau {\eps}/2)\bigl(x_d+a_+-t c\varepsilon \varphi\bigr)^+
\]
where \(c=c(d)\) is a small constant chosen such that for all \(x\in B_{\sfrac{1}{100}}(Q)\) and \( t\in [0,1)\),
\begin{equation}\label{e:okinb}
\begin{split}
u(x)&\le \lambda_+(1-\tau \eps)\bigl(x_d+a_+\bigr)^+
\\
&\le  \lambda_+(1-\tau \eps/2)\bigl(x_d+a_+-c\eps\bigr)^+ <f_t(x)
\end{split}
\end{equation}
where we have used that \((x_d+a_+\bigr)\) is within two universal constant for for \(x\in B_{\sfrac{1}{100}}(Q)\).

We now let \(\bar t\in (0,1]\) the largest \(t\) such that \(f_t\ge u\) in \(B_1\) and we claim that \(\bar t=1\). Indeed assume that \(\bar t<1\), then there exists \(\bar x\in B_1\) such that 
\begin{equation}\label{e:tocca}
 u(x)-f_{\bar t}(x)\le u(\bar x)-f_{\bar t}(\bar x)=0\quad\text{for all}\quad x\in B_1. 
\end{equation}
Note that by \eqref{e:okinb}, \(\bar x \notin B_{\sfrac{1}{100}}(Q)\), while, by \ref{i:P1} and \eqref{e:inizio},   \(\bar x \in \{\varphi>0\}\). Moreover  \(\bar x \in \{f_{\bar t}=0\}\), indeed otherwise, by   \ref{i:P2},  \(\Delta f_{\bar t}(\bar x)<0\) and \(\Delta u(\bar x)=0\), a contradiction with \eqref{e:tocca}. Assume now  \(\bar x \in \{f_{\bar t}=0\}\), since \(u\) is a viscosity solution we get that, by \ref{i:P3}, 
\[
\lambda_+^2\le \abs{\nabla f_{\bar t}(\bar x)}^2= \lambda_+^2(1-\tau \eps/2)^2-2c \eps \bar t  \lambda_+ \partial_d \varphi(\bar x)+O(\eps^2)<\lambda^2_+
\]
provided \(\eps\le \bar \eps(d, \lambda_+)\ll 1\) (note that necessarily  \(u (\bar x)=0\) and thus \(\bar x \in \{|x_d|\le \sfrac{1}{100}\}\)). This  contradiction implies that \(\bar t=1\). Hence, by \ref{i:P4}, we get for all \(x\in B_{\sfrac{1}{6}}\).
\[
u(x)\le  \lambda_+(1-\tau \eps/2)\bigl(x_d+a_+-c\eps\varphi \bigr)^+\le \lambda_+\bigl(x_d+a_+-\bar c\eps\bigr)^+ 
\]
for a suitable dimensional constant \(\bar c\). Setting 
\[
\bar a_+=a_+-\bar c\eps,\qquad \bar b_+=b_+ 
\]
and recalling that \(\eps=(b_+-a_+)\) allows to conclude the proof in this case.

\medskip
\noindent
 \(\bullet\) \emph{Case 2. Improvement from below.} We now assume that, at the point $P$, $u^+$ is closer to $\lambda_+(2+a_+)^+$ than to  $\lambda_+(2+b_+)^+$. Hence, we have
\[
u^+(P)\ge \lambda_+(2+b_+)^++\frac{\lambda_+ (a_+-b_+)}{2}.
\]
and we set again 
\[
\eps:=a_+-b_+ \le \bar \eps.
\]
Arguing as in Case \(1\), by  \cref{lm:improvement}, there exists a dimensional constant \(\tau\) such that 
\begin{equation}\label{e:inizio1}
u^+\ge \lambda_+(1+\tau \eps)\bigl(x_d+b_+\bigr)^+\quad \text{in}\quad  B_{1}.
\end{equation}

We need now  to distinguish two further sub-cases:

\medskip
\noindent
 \(\bullet\) \emph{Case 2.1:} Suppose that 
 \[
 0\le b_+-b_-\le \eta \eps
 \]
 where \(\eta\ll \tau\) is a small  universal constant which we will choose at the end of the proof. In this case, for \(x\in B_1\),
 \begin{equation}\label{e:pronti}
 \begin{split}
 u&\ge \lambda_+(1+\tau \eps)\bigl(x_d+b_+\bigr)^+-\lambda_-(x_d+b_-\bigr)^-
 \\
 &\ge \lambda_+(1+\tau \eps)\bigl(x_d+b_+\bigr)^+-\lambda_-(1-c_1\eta\eps)\bigl(x_d+b_+\bigr)^-
 \end{split}
 \end{equation}
for a suitable universal constant \(c_1\). We now take \(\varphi\) as in \eqref{e:phi} and we set, for \(t\in [0,1]\),
\[
f_t(x)=\lambda_+(1+\tau \eps/2)\bigl(x_d+b_++c_2t\varphi\bigr)^+-\lambda_-(1-c_1\eta\eps)\bigl(x_d+b_++c_2t\varphi\bigr)^-.
\]
for a suitably small universal constant \(0<c_2\ll \tau\), chosen so that for all \(x \in B_{\sfrac{1}{100}}(Q)\): 
\[
\bigl(1+\tau \eps\bigr)\bigl(x_d+b_+\bigr)^+\ge \bigl(1+\tau \eps/2\bigr)\bigl(x_d+b_++c_2\eps \bigr)^+.
\]
This together with  \eqref{e:inizio1} implies that 
\begin{equation}\label{e:contr}
\begin{split}
u(x)\ge  \lambda_+(1+\tau \eps)\bigl(x_d+b_+\bigr)^+&\ge \lambda_+(1+\tau \eps/2)\bigl(x_d+b_++c_2 \bigr)^+
\\
&\ge f_1(x)\ge f_t(x)\qquad \text{for all \(x \in B_{\sfrac{1}{100}}(Q)\), \(t\in [0,1]\).}
\end{split}
\end{equation}
Furthermore  \(u\ge f_0\) in \(B_1\) thanks to \eqref{e:pronti}.

As in Case 1 we let \(\bar t\) the biggest \(t\) such that \(f_t\le u\) in \(B_1\) and \(\bar x\) the first contact point, so that 
\begin{equation*}
 u(x)-f_{\bar t}(x)\ge u(\bar x)-f_{\bar t}(\bar x)=0\qquad\text{for all \(x\in B_1\)}. 
\end{equation*}
 Since \(\Delta f_{\bar t}>0\)  on  \(\{f_t\ne 0\}\cap B_{\sfrac{1}{100}}(Q)\), as in Case 1, \(\bar x\) is a free boundary point. Moreover, since \(f_{\bar t}\) changes sign in a neighborhood of \(\bar x\):
\begin{align*}
&\text{either}&&\bar x \in \Gamma_{\textsc{op}}^+=\partial \Omega_u^+\setminus \partial \Omega_u^-,
\\
&\text{or }&&\bar x \in \Gamma_{\textsc{tp}}=\partial \Omega_u^+\cap \partial \Omega_u^-.
\end{align*}
In the first  case, by definition of viscosity solution and \ref{i:P3},
\[
\lambda_+^2\ge |\nabla f_{\bar t}^+(\bar x)|^2= \lambda_+^2(1+\tau \eps/2)^2+2c \eps \bar t  \lambda_+ \partial_d \varphi(\bar x)+O(\eps^2)>\lambda_+^2,
\]
a contradiction for \(\eps\ll1\). In the second case we have a contradiction as well, provided \(\eta \ll \tau\), since (recall also that \(\lambda_+\ge \lambda -\),  \eqref{e:ordine}):
\[
\begin{split}
\lambda_+^2-\lambda_-^2&\ge |\nabla f_{\bar t}^+|^2- |\nabla f_{\bar t}^-|^2
\\
&= \lambda_+^2(1+\tau \eps/2)^2-\lambda_-^2(1-c_1\eta \eps)^2+2c_2 \eps \bar t  (\lambda_+-\lambda_-) \partial_d \varphi(\bar x)+O(\eps^2)
\\
&>\lambda_+^2-\lambda_-^2
\end{split}
\]
provided \(\eta =\eta (d) \ll \tau\) and \(\eps \ll 1\) (only depending on \(d\) and \(\lambda_+\)). Hence, \(\bar t=1\),  \(u\ge f_1\) which  implies the desired conclusion by setting 
\[
\bar a_+=a_+,\qquad \bar b_+=b_+ +\bar c_2\eps
\]
and by recalling that \(\eps=(a_+-b_+)\).

\medskip
\noindent
 \(\bullet\) \emph{Case 2.2:} Assume instead that:
 \[
 b_+-b_-\ge \eta \eps
 \]
 where \(\eta=\eta(d)\) has been chosen according to Case 2.1. In this case we consider the family of functions 
 \[
 f_t(x)=\lambda_+(1+\tau \eps/2)\bigl(x_d+b_++\eta t\varphi\bigr)^+-\lambda_-\bigl(x_d+b_-)^-.
 \]
 Being  \(\varphi \le 1\), this is well defined since \(b_+\ge b_-+\eta\). Moreover \(u\ge f_0\) and,  thanks, to  \eqref{e:inizio1} and   by possibly choosing  \(\eta\) smaller depending only on the dimension, 
 \[
 u(x)\ge f_{1}(x)\ge  f_t(x) \qquad \text{for all \(x \in B_{\sfrac{1}{100}}(Q)\), \(t\in [0,1]\).}
 \] 
  We consider again the first touching time  \(\bar t\) and the first touching point \(\bar x\). Note that this can not happen where \(u\ne 0\). Moreover, by the very definition of  \(f_{\bar t}\), \(\bar x \in  \partial \Omega^+_u\setminus \partial \Omega_u^-\). However, again by arguing as in Case 2.1, this is in contradiction with \(u\) being a  viscosity solution. We now conclude as in the previous cases.
  
Since  either the assumption of Case \(1\) or the one of Case \(2\) is always satisfied, this concludes the proof.
\end{proof}

The next lemma deals with the case in which the origin is not a branching point.

\begin{lemma}[Partial Boundary Harnack II]\label{lm:partialharnack1}
Given \(L\ge \lambda_+\ge \lambda_->0\) there exist constants $\bar\eps=\bar \eps(d, \lambda_{\pm},L)>0$, \(M=M(d, \lambda_{\pm},L)\) and  $c=c(d, \lambda_{\pm},L)\in(0,1)$  such that for every function $u:B_4\to \R$ satisfying (a), (c) and (d) in \cref{t:eps_regularity}  the following property holds true. 
If there are constants  \(a,b \in \bigl(-\sfrac{1}{100}, \sfrac{1}{100}\bigr)\) with 
\[
0\le a-b\le \bar \eps
\] 
such that for \(x \in B_4\)
\[
H_{\alpha, \bo{e}_d}(x+b \bo{e}_d)\le u(x)\le H_{\alpha, \bo{e}_d}(x+a \bo{e}_d)
\]
and
\[
\lambda_++M\eps \le \alpha \le 2L,
\]
 then there are  constants  \(\bar a,\bar b \in \bigl(-\sfrac{1}{100}, \sfrac{1}{100}\bigr)\) with 
 \[
 0\le \bar b-\bar a\le c (b-a)
 \]
such that  for \(x \in B_{\sfrac{1}{6}}\)
\[
H_{\alpha, \bo{e}_d}(x+\bar b \bo{e}_d)\le u(x)\le H_{\alpha, \bo{e}_d}(x+\bar a \bo{e}_d).
\]
\end{lemma}

\begin{proof} 
 We consider the point \(P=( 0, \dots, 0, 2)\) and we distinguish the two  cases (note that one of the two is always satisfied):
\begin{align*}
\text{either}& &H_{\alpha, \bo{e}_d}\Bigl(P+b \bo{e}_d\Bigr)+\frac{\alpha (a-b)}{2} \le &u(P),
\\
\text{or}&  &H_{\alpha, \bo{e}_d}\Bigl(P+a \bo{e}_d\Bigr)-\frac{\alpha (a-b)}{2} \ge &u(P).
\end{align*}
Since the argument in both cases is completely symmetric we only consider the second one. If we set 
\[
\eps=(a-b),
\]
 by \cref{lm:improvement}  and by arguing as in \cref{lm:partialharnack2} we deduce the existence of a dimensional constant \(\tau\) such that 
\[
u\le \alpha(1-\tau \eps) \bigl (x_d+a\bigr)^+-\beta \bigl (x_d+a\bigr)^-
\]
in \(B_1\). We let \(\varphi\) as in \eqref{e:phi} and we set 
\[
f_t(x)=\alpha(1-\tau \eps/2) \bigl (x_d+a-ct \varphi\bigr)^+-\beta \bigl (x_d+a-ct \varphi\bigr)^-
\]
 where \(c\) is a dimensional constant chosen such that 
 \[
 u(x)\le f_{1}(x)\le  f_t(x) \qquad \text{for all \(x \in B_{\sfrac{1}{100}}(Q)\), \(t\in [0,1]\).}
 \] 
 where, again, \(Q=( 0, \dots, 0, \sfrac{1}{5})\). As in \cref{lm:partialharnack2} we let \(\bar t\) and \(\bar x\) be the first contact time and the first contact point and we aim to show that \(\bar t=1\). For, we note that, by the same arguments as in \cref{lm:partialharnack2},  necessarily \(\bar x \in \{u=0\}\). We claim that 
 \[
 \bar x \in \Gamma_{\textsc{tp}}=\partial \Omega_u^+\cap \partial \Omega_u^-.
 \]
 Indeed otherwise \(\bar x \in \partial \Omega^{-}_{u}\setminus \partial \Omega^{+}_{u}\), the case  \(\bar x \in \partial \Omega^{+}_{u}\setminus \partial \Omega^{-}_{u}\) being impossible since \(f_{\bar t}\) is negative in a neighborhood of \(\bar x\). By definition of viscosity solution this would imply
\begin{equation}\label{e:caldo}
\lambda_{-}^2 \ge |\nabla f_{\bar t}^-(\bar x)|^2=\beta^2-O(\eps)\ge \lambda_{-}^2 +2M\lambda_+ \eps -O (\eps) 
\end{equation}
where the implicit constants in \(O(\eps)\) depends on \(\lambda_{\pm}\), \(L\) and \(d\) and we exploited that, since  \(\alpha \ge \lambda_++M\eps\), 
\[
\beta^2=\alpha^2-\lambda_+^2+\lambda_-^2\ge \lambda_-^2+2M\lambda_+ \eps.
\]
Inequality \eqref{e:caldo} is impossible if \(M\) is chosen sufficiently large. Hence \(\bar x \in \Omega^{-}_{u}\cap \partial \Omega^{+}_{u}\). This however implies:
\[
\begin{split}
\lambda_+^2-\lambda_{-}^2 &\le |\nabla f_{\bar t}^+(\bar x)|^2-|\nabla f_{\bar t}^-(\bar x)|^2
\\
&= \alpha^2(1-\tau \eps/2)^2-\beta^2-2c \bar t\eps(\alpha-\beta) \partial_d \varphi (\bar x)  +  O(\eps^2)
\\
&\le \lambda_+^2-\lambda_{-}^2-\alpha^2\tau  \eps+ O(\eps^2).
\end{split}
\]
where we have used \ref{i:P3}, the equality
\[
\lambda_+^2-\lambda_{-}^2=\alpha^2-\beta^2
\]
and that since \(\lambda_+\ge \Lambda_-\), \(\alpha\ge \beta\).  This is a contradiction  provided \(\bar \eps\) is chosen small enough. Hence \(\bar t=1\) and , as in 
\cref{lm:partialharnack2}, this concludes the proof.
\end{proof}

With \cref{lm:partialharnack1,lm:partialharnack2} at hand we can use the same arguments as in  \cite{desilva,dfs}  to prove \cref{cor:comp}.

\begin{proof}[\bf Proof of \cref{cor:comp}]
	We distinguish two cases:

\medskip

\noindent
\( \bo{0\le \ell<+\infty:}\) By triangular inequality we have
$$
\|u_k-H_{\lambda_+, \bo{e_d}}\|_{L^\infty(B_1)}\leq (2\ell+1) \eps_k 
$$
for $k$ sufficiently large. In particular we can repeatedly apply \cref{lm:partialharnack2} as in \cite{desilva}, see also \cite[Lemma 7.14 and Lemma 7.15]{notes} for a detailed proof, to deduce that if we define the sequence $(w_k)_k$ by
\begin{equation*}\label{e:wk}
w_k(x)=
\begin{cases}
w_{+,k}(x):=\dfrac{u_k(x)-\lambda_+x_d^+}{\alpha_k\eps_k}\qquad&x\in \Omega_{u_k}^+\cap B_1
\\
w_{-,k}(x):=\dfrac{u_k(x)+\lambda_- x_d^-}{\beta_k\eps_k}\qquad&x\in \Omega_{u_k}^-\cap B_1
\end{cases}
\end{equation*} 
then the sets
\[
\tilde\Gamma_k^\pm:=\Big\{(x, w_{\pm,k}(x))\ :\ x\in \overline{\Omega_{u_k}^\pm\cap B_{\sfrac12}}\Big\},
\]
converge, up to a not relabeled subsequence,  in the Hausdorff distance to the closed graphs
\[
\tilde \Gamma_\pm=\Big\{(x, w_\pm(x))\ :\ x\in \overline{B_{\sfrac12}^\pm}\Big\}.
\]
where $w\in C^{0,\alpha}$ for a suitable $\alpha$. Since 
$$
h_k(x):=\frac{H_{\alpha_k,\bo{e_d}}-H_{\lambda^+, \bo{e_d}}}{\eps_k} \to 
\begin{cases}
\lambda_+^{-1} \ell x_d & \text{if } x_d>0\\
\lambda_- ^{-1}\ell x_d & \text{if } x_d<0\,,
\end{cases} 
$$ 
the original sequence \(v_k\) satisfies that their graphs,
\[
\tilde \Gamma_\pm=\Big\{(x, v_\pm(x))\ :\ x\in \overline{B_{\sfrac12}^\pm}\Big\},
\]
converges to the graph of a  limiting function \(v\) as we wanted, this in particular proves (i),  (ii)  and (iii).

Since \(0\in \partial \Omega_{u_k}^+\cap \partial \Omega_{u_k}^-\) then \(0\) is in the domain of \(v_{\pm,k}\) and 
\[
v_{\pm,k}(0)=0
\]
which implies that \(v_{\pm}(0)=0\). To show that \(v_+(x)\le v_-(x)\) for  \(x=(x',0)\in \{x_d=0\} \cap B_{\sfrac{1}{2}}\) we simply exploit (iii) at the points \(x^\pm_k=(x', t_k^\pm)\) where 
\[
t^+_k=\sup\bigl\{t: (x',t)\in   \partial \Omega_{u_k}^+\bigr\}
\qquad \text{and} \qquad 
t^-_k=\inf\bigl\{t: (x',t) \in \partial \Omega_{u_k}^-\bigr\}
\]
and by noticing that \(-t_k^+ \le -t_k^-\). Finally to show the last claim it is enough to note that if \(x_k\in \partial \Omega_{u_k}^+\cap \partial \Omega_{u_k}^-\) is converging to \(x\)  then  \(v_{+,k}(x_k)=v_{-,k}(x_k)\) and thus \(v_{+}(x)=v_{-}(x)\), yielding \(x\in \mathcal C\).

\medskip

\noindent
\( \bo{\ell=\infty:}\)  In this case the conclusion follows exactly as in \cite{dfs} by using \cref{lm:partialharnack1} and noticing that its assumptions are satisfied since $\ell=\infty$.	
\end{proof}

\subsection{The linearized problem: proof of \cref{l:limitprob}}\label{sub:linearized_problem}

The following technical lemma is instrumental to the proof of \cref{l:limitprob}. We defer its proof to \cref{s:app} below.

\begin{lemma}\label{l:touching}
Let  $u_k$, \(\eps_k\)  and \(\alpha_k\) be as in the statement of \cref{cor:comp},  \(v_k\) be defined by \eqref{e:vk} and    \(v_\pm\) be as in  \cref{cor:comp}.  Then:

\begin{enumerate}[label=\textup{(\arabic*)}]
\item Let $P_+$ a strictly subharmonic (superharmonic) function on $B_{\sfrac12}^+$   touching
 $v_+$ strictly from below (above) at a point $x_0\in\{x_d=0\}\cap B_{\sfrac12}$. Then, there exists a sequence of points $\partial \Omega_{u_k}^+\ni x_k\to x_0$ and a sequence of comparison functions $Q_{k}$ such that $Q_{k}$  touches from below (above) $u^+_k$ at $x_k$, and such that 
\begin{equation}\label{e:sticazzi1}
\nabla Q_{k}^+(x_k)=\alpha_k\bo{e}_d+\alpha_k\eps_k\nabla P_+(x_0)+o(\eps_k).
\end{equation}
\item Let $P_-$ be a strictly subharmonic (superharmonic) function on $B_{\sfrac12}^-$ and  touching $v_-$ strictly  from below (above) at a point $x_0\in\{x_d=0\}\cap B_{\sfrac12}$. Then, there exists a sequence of points $\partial \Omega_{u_k}^-\ni x_k\to x_0$ and a sequence of comparison functions $Q_{k}$ such that $Q_{k}$ touches from below (above)  $-u^-_k$ at $x_k$, and such that 
\begin{equation}\label{e:sticazzi2}
\nabla Q_{k}^-(x_k)=-\beta_k\bo{e}_d+\beta_k\eps_k\nabla P_-(x_0)+o(\eps_k).
\end{equation}
\item Let  \(p,q\in \R\) and \(\tilde P\) be a function on   $B_{\sfrac12}$ such that \(\partial_d \tilde P=0\).  Suppose that \(\tilde{P}\) is subharmonic (superharmonic) and  that the function
\[
P:=px_{d}^{+}-qx_{d}^{-}+\tilde{P}
\]
touches \(v\) strictly from below (above) at a point \(x_0\in \mathcal C\). Then, there exists a sequence of points  $x_k\to x_0$ and  a sequence of comparison functions $Q_{k}$ such that $Q_{k}$ touches from below (above) the function $u_k$ at $x_k\in\partial\Omega_{u_k}$, and such that 
\begin{equation}\label{e:sticazzi3}
\begin{split}
\nabla Q_{k}^+(x_k)&=\alpha_k\bo{e}_d+\alpha_k\eps_kp+o(\eps_k)
\\
 \nabla Q^-_{k}(x_k)&=-\beta_k\bo{e}_d+\beta_k\eps_k q+o(\eps_k).
 \end{split}
\end{equation}
In particular, if  \(p>0\) and \(Q_{k}\) touches \(u_{k}\) from below then \(x_{k}\notin \partial \Omega_{u_{k}}^{-}\setminus \partial \Omega_{u_{k}}^{+}\), while if  \(q<0\) and $Q_k$ touches $u_k$ from above then \(x_{k}\notin \partial \Omega_{u_{k}}^{+}\setminus \partial \Omega_{u_{k}}^{-}\).
\end{enumerate}
\end{lemma}

\begin{proof}[Proof of \cref{l:limitprob}]
We note that \(v_k^\pm\) converge uniformly to \(v_\pm\) on every compact subset of \(\{\pm x_d>0\}\cap B_{\sfrac{1}{2}}\). Since these functions are harmonic there, by elliptic estimates the convergence is smooth and in particular \(v_{\pm}\) are harmonic on the (open) half balls \(B^\pm_{\sfrac{1}{2}}\). Hence we only have to check the boundary conditions on \(\{x_d=0\}\). We distinguish two cases. 

\medskip
\noindent$\bo{\ell=\infty}.$
In this case we first want to show that \(\mathcal J = \emptyset\). Assume not, since  the set $\{v_->v_+\}$ is open in \(\{x_d= 0\}\),  it contains a $(d-1)$-dimensional ball 
$$
B'_\eps(y'):=B_\eps((y',0))\cap \{x_d=0\}\subset \mathcal J\,.
$$
Next let $P$ be the polynomial 
\[
P(x)=A\big((d-\sfrac12)x_d^2-|x'-y'|^2\big)-Bx_d\,,\quad\text{where}\quad x=(x',x_d)\,,
\]
for some constants $A,B$. We first choose \(A\gg1 \) large enough so that  
\[
P<v^+\qquad \text{on} \quad \{|x'-y'|=\eps\}\cap \{x_d=0\} 
\]
and then we choose \(B\gg A\) so that 
 \[
 P<v^+ \quad \text{on} \quad B_\eps((y',0)).
 \]
Now we can translate  \(P\) first down and then up to find that there exists \(C\) such that \(P+C\) is touching \(v^+\) from below at a point \(x_0\in B_\eps((y',0)) \cap  \{x_d\ge 0\}\). Since \(\Delta P>0\), the touching point can not be in the interior of the (half) ball and thus \(x_0\in B'_\eps(y')\subset \mathcal J\).

By using \cref{l:touching}, there exists  a sequence of points $\partial \Omega_{u_k}^+\ni x_k\to x_0$ and of functions $Q_k$ touching $u^+_k$ from below at $x_k$ and such that 
\[
\nabla Q_k^+(x_k)=\alpha_k\bo{e}_d+\alpha_k\eps_k\nabla P(x_0)+o(\eps_k).
\]
Since \(x_0\in \mathcal J\), by  \eqref{e:nodp} in \cref{l:touching}, $x_k\in\partial\Omega_{u_k}^+\setminus\partial\Omega_{u_k}^-$. Hence, by (ii) in \cref{l:NY}
\[
\lambda_+^2\ge |\nabla Q_k^+(x_k)|^2\ge \alpha_k^2+2\alpha_k^2\eps_k\partial_d P(x_0)+o(\eps_k)
\]
Hence, recalling the definition of \(\ell\),
\[
-B=\partial_d P(x_0)\le \frac{\lambda_+^2-\alpha_k^2}{2\alpha_k^2\eps_k}+o(1)\to -\infty.
\]
This contradiction proves that \(\mathcal J=\emptyset\).

We next prove the transmission condition in \eqref{e:trasmission}. Let us show that 
\[
\alpha_\infty^{2}\partial_{d}v_{+}-\beta_\infty^{2}\partial_{d} v_{-}\le 0,
\]
 the opposite inequality can then be proved by the very same argument. Suppose that there exist \(p\) and \(q\) with \(\alpha_\infty^{2} p>\beta_\infty^{2}q\) and a strictly sub-harmonic function  \(\tilde P\) with \(\partial_{d} \tilde P=0\)   such that 
\[
P=px_{d}^{+}-qx_{d}^{-}+\tilde{P}
\]
touches \(v\) strictly  from below at a point \(x_{0}\in \{x_{d}=0\}\cap B_{\sfrac 12}\) (note that the last set coincide with \(\mathcal C\) by the previous step). By \cref{l:touching} there exists a sequence of points  \(\partial \Omega_{u_{k}}\ni x_{k}\to x_{0}\) and a sequence of comparison functions \(Q_{k}\) touching \(u_{k}\) from below at \(x_{k}\) and satisfying \eqref{e:sticazzi3}. In particular \(x_{k}\notin \partial \Omega_{u_{k}}^{-}\setminus \partial \Omega_{u_{k}}^{+}\). We claim that  \(x_{k}\in \partial \Omega_{u_{k}}^{+}\cap \partial \Omega_{u_{k}}^{-}\). Indeed, otherwise  by (A.1) in \cref{l:optimality}, 
\[
\lambda_+^2\ge |\nabla Q_k^+(x_k)|^2
\]
and, by arguing as above,  this  contradicts  \(\ell=+\infty\). Hence, by \cref{l:optimality} (A.3)
\[
\begin{split}
\lambda_{+}^{2}-\lambda_{-}^{2}&\ge |\nabla Q^{+}_{k}(x_{k})|^{2}-|\nabla Q^{-}_{k}(x_{k})|^{2}
\\
&=\alpha_{k}^{2}-\beta_{k}^{2}+2\eps_{k}(\alpha_{k}^{2}p-\beta_{k}^{2}q)+o(\eps_{k})
\\
&=\lambda_{+}^{2}-\lambda_{-}^{2}+2\eps_{k}(\alpha_{k}^{2}p-\beta_{k}^{2}q)+o(\eps_{k}).
\end{split}
\]
Dividing by \(\eps_{k}\) and letting \(k\to \infty\), we obtain the desired contradiction.

\medskip
\noindent$\bo{0\leq \ell <\infty}.$
We start by  showing that \(\lambda_{\pm}^2\partial_d v_\pm \ge -\ell\)  on \(B_{\sfrac{1}{2}}\cap \{x_d=0\}\). We focus on \(v_-\) since the argument is symmetric. Let us assume that there exists \(q\in \R\) with  \(\lambda_-^2q<-\ell\) and a strictly subharmonic function  \(\tilde P\) with \(\partial_{d} \tilde P=0\)  such that function 
\[
P=q x_{d}+\tilde P
\]
touches \(v_-\) strictly  from below  at a point  \(x_{0}\in \{x_{d}=0\}\cap B_{\sfrac 12}\). Let now \(x_{k}\) and \(Q_{k}\) be as in   \cref{l:touching} (2). By the optimality conditions
\[
\lambda_{-}^{2}\le |\nabla Q_{k}^{-}(x_{k})|^{2}=\beta_{k}^{2}+2\eps_{k}\beta_{k}^{2}q+o(\eps_{k}).
\]
Since $\ell<\infty$, we have \(\beta_{k}=\lambda_{-}+O(\eps_{k})\) and so the above inequality leads to
\[
-\frac{\ell}{\lambda_-^2}=\lim_{k\to \infty}\frac{\lambda_{-}^{2} - \beta_{k}^{2} }{2\eps_{k}\beta_{k}^{2}} \le q<-\frac{\ell}{\lambda_-^2}
\]
which is a contradiction. 

We now show that   \(\lambda_\pm^2 \partial_d v_\pm =-\ell\) on \(\mathcal J\) and again we focus on \(v_-\). By  the previous step   it is enough to show that if  there exists a strictly superharmonic polynomial \(\tilde P\) with \(\partial_{d} \tilde P=0\) such 
\[
P=q x_{d}+\tilde P
\]
 touches \(v_{-}\) strictly  from above  at a point \(x_{0}\in \mathcal J\), then \(\lambda_-^2q\le -\ell\). Again, by \cref{l:touching}, we find points \(x_{k}\to x_{0}\) and functions \(Q_{k}\)  satisfying \eqref{e:sticazzi2} and touching \(-u^-_{k}\) from below at \(x_{k}\). Since \(x_{0}\in \mathcal J\), by \eqref{e:nodp} in \cref{cor:comp}, \(x_{k}\in \partial \Omega_{u_k}^-\setminus  \partial \Omega_{u_k}^+\). Hence, by \cref{l:optimality}, 
\[
\lambda_{-}^{2}\ge |\nabla Q_{k}^{-}(x_{k})|^{2}=\beta_{k}^{2}+2\beta_{k}^{2} \eps_{k} q+o(\eps_{k}),
\]
which  by arguing as above implies that \(\lambda_-^2 q\le -\ell\). 

It then remain to show the transmission condition in \eqref{e:twomembrane} at points in \(\mathcal C\). Again by symmetry of the arguments we will only show that 
\[
\lambda_{+}^{2}\partial_{d}v_{+}-\lambda_{-}^{2}\partial_{d} v_{-}\le 0,
 \qquad \text{on \(\mathcal C\)}.
\]
Let us hence assume that there exist  \(p\) and \(q\) with \(\lambda_{+}^{2} p>\lambda_{-}^{2}q\) and a strictly subharmonic polynomial \(\tilde P\) with \(\partial_{d} \tilde P=0\)   such that 
\[
P=px_{d}^{+}-qx_{d}^{-}+\tilde{P}
\]
touches \(v^+ \) and \(v^-\) strictly  from below at \(x_{0}\in \mathcal C\). By \cref{l:touching}, we find points \(x_{k}\to x_{0}\) and functions \(Q_{k}\) satisfying \eqref{e:sticazzi3}. In particular  \(x_{k}\notin \partial \Omega_{u_{k}}^{-}\setminus \partial \Omega_{u_{k}}^{+}\). By the previous step  we know that \(\lambda_-^2q\ge -\ell\) and thus \(\lambda_+^2 p>-\ell\), since we are assuming $\lambda_+^2p+>\lambda_-^2q\geq 0$. We now  distinguish two cases:
\begin{itemize}
\item[1)] \(x_{k }\) are one-phase points, namely \(x_{k}\in \partial \Omega_{u_{k}}^{+}\setminus \partial \Omega_{u_{k}}^{-}\). In this case 
\[
\lambda_{+}^{2}\ge  |\nabla Q_{k}^{+}(x_{k})|^{2}=\alpha_{k}^{2}+2\alpha_{k}^{2} \eps_{k} p+o(\eps_{k}),
\]
which implies that 
\[
\lambda_+^2 p+\ell=\lambda_{+}^2\lim_{k \to \infty}\Bigl(p+\frac{\alpha_k^2-\lambda_+^2}{2\alpha_k^2\eps_k}\Bigr)\le 0 
\]
 in contradiction with \(\lambda_+^2p>-\ell\).
\item[2)] \(x_{k }\) are two-phase points, namely \(x_{k}\in \partial \Omega_{u_{k}}^{+}\cap\partial \Omega_{u_{k}}^{-}\). Arguing as in Case 1,  we have that, by \cref{l:optimality},
\[
\begin{split}
\lambda_{+}^{2}-\lambda_{-}^{2}&\ge |\nabla Q^{+}_{k}(x_{k})|^{2}-|\nabla Q^{-}_{k}(x_{k})|^{2}
\\
&=\alpha_{k}^{2}-\beta_{k}^{2}+2\eps_{k}(\alpha_{k}^{2}p-\beta_{k}^{2}q)+o(\eps_{k})
\\
&=\lambda_{+}^{2}-\lambda_{-}^{2}+2\eps_{k}(\lambda_+^2p-\lambda_-^2q)+o(\eps_{k})
\end{split}
\]
which gives a contradiction with $ \lambda_+^2p>\lambda_-^2q$, as \(\eps_k \to 0\).
\end{itemize}
\end{proof}

\subsection{Proof of \cref{l:big_regime,l:small_regime}}\label{sub:proof_iof}
We recall the following regularity results for the  limiting problems.

\begin{lemma}[Regularity for the transmission problem]\label{l:reg_trans}
There exists a universal constant $C=C(\alpha_\infty, \beta_\infty, d)>0$ such that if  $v\in C^0(B_{\sfrac12})$ is  a viscosity solution of \eqref{e:trasmission} with \(\|v\|_{L^\infty(B_{\sfrac12})}\le 1\) then there exists \(\bo{v}\in\mathbb \R^{d-1}\), \(p, q \in \R\) with  $\alpha^2_\infty \,p=\beta^2_\infty \, q$ such that 
\begin{equation}\label{e:decay_trans}
\sup_{x\in B_r}
\frac{\bigl|v(x)-v(0)-(\bo{v}\cdot x'+ p\,x_d^+-\,q\,x_d^-)\bigr|}{r^2} \leq C 
\end{equation}
\end{lemma}

The proof of this fact can be found in \cite[Theorem 3.2]{dfs}. A similar result holds for the linearized problem \eqref{e:twomembrane}.

\begin{lemma}[Regularity for the two-membrane problem]\label{l:reg_twom}
There exists a universal constant $C=C(\lambda_\pm, d)>0$ such that if  $v\in C^0(B_{\sfrac12})$ is  a viscosity solution of \eqref{e:twomembrane} with \(\|v\|_{L^\infty(B_{\sfrac12})}\le 1\) then there exists \(\bo{v}\in\mathbb \R^{d-1}\), \(p, q \in \R\) satisfying $\lambda_+^2 \,p=\lambda_-^2 \, q\ge -\ell$    such that 
\begin{equation}\label{e:decay_twomemb}
\sup_{x\in B_r} \frac{ \bigl|v(x)-v(0)-(\bo{v}\cdot x'+  \,p\,x_d^+-q\,x_d^-)\bigr|}{r^{\sfrac{3}{2}}}\leq C(1+\ell)
\end{equation}	
\end{lemma} 

The proof of the above lemma reduces easily to the one of the thin obstacle problem,  since we were not able to find the statement of this  fact in the literature, we sketch its proof in \cref{app:twomembrane}.

It is by now well known that the regularity theory fo the limiting problems and a classical compactness argument  prove  \cref{l:big_regime,l:small_regime}. We sketch their arguments here:

\begin{proof}[\bf Proof of \cref{l:small_regime}] We argue by contradiction and we assume that for fixed \(\gamma\in (0,1/2)\) and \(M\) we can find a sequences of  functions \(u_k\)  and  numbers \(\alpha_k\) such that  
\[
\eps_k=\norm{u_k-H_{\alpha_k,\bo{e}_d}}_{L^\infty(B_1)}\to 0 \,,\quad\text{and}\quad  0\le \alpha_k-\lambda_+\le  M \eps_k,
\]
but for which  \eqref{e:iof_normal_oscillation} and \eqref{e:iof_thesis} for any choice of \(\rho\) and \(C\).  Note that by the second assumption above
\[
\ell<\frac{M}{\lambda_+}
\]
We let  \((v_k)_k\) be the sequence of functions defined in \eqref{e:vk} and we assume that they converge to a function \(v\) as in  \cref{cor:comp}, note that \(\|v\|_{L^\infty(B_{\sfrac12})}\le 1\). By \cref{l:limitprob}, \(v\) solves \eqref{e:twomembrane} and thus by  \cref{l:reg_twom} there exists \(\bo{v}\in\mathbb \R^{d-1}\), \(p, q \in \R\) satisfying $\lambda_+^2 \,p=\lambda_-^2 \, q\ge -\ell$    such that for all \(r\in (0,1/4)\)
\begin{equation}
\sup_{x\in B_{\rho}} \frac{ \bigl|v(x)-v(0)-(\bo{v}\cdot x'+  \,p\,x_d^+-q\,x_d^-)\bigr|}{r^{\gamma}}\leq r^{\sfrac{3}{2}-\gamma}C(1+M)
\end{equation}	
Hence we can fix \(\rho=\rho(\lambda_{\pm}, \gamma, M)\) such that 
\begin{equation}\label{e:rho}
\sup_{x\in B_{\rho}}  \bigl|v(x)-v(0)-(\bo{v}\cdot x'+  \,p\,x_d^+-q\,x_d^-)\bigr|\leq \frac{\rho^\gamma}{2}.
\end{equation}
We now set  
\begin{equation*}\label{e:rotation}
\tilde{\alpha}_k:=\alpha_k(1+ \eps_k p)+\delta_k \eps_k
\qquad \mbox{and}\qquad 
\bo{e}_k:=\frac{\bo{e}_d +\eps_k \bo{v}}{\sqrt{1+\eps_k^2\,|\bo{v}|^2}}\,,
\end{equation*}
where \(\delta_k\to 0\) is chosen so that \(\tilde \alpha_k \ge \alpha_k\),  note that the existence of such a sequence is due to the condition \(\lambda_+^2 p\ge -\ell\) since 
\[
\alpha_k(1+\eps_k p)=\Bigl(\lambda_++\frac{\ell}{\lambda_+}\eps_k+o(\eps_k)\Bigr)(1+\eps_k p)\ge \lambda_++o(\eps_k).
\]
  We let  \(H_k:=H_{\tilde \alpha_k, \bo{e_k}}\) and we note  that 
\[
|\alpha_k-\alpha|+|\bo{e_k}-\bo{e}_d|\leq C\, \eps_k\,,
\]
for a universal constant $C>0$, hence the proof will be concluded if we can show that 
\[
\sup_{B_{\rho}} |u_k(x)-H_k(x)|\le \rho^{\gamma} \eps_k
\]
where \(\rho\) is defined so that \eqref{e:rho} holds. This however easily follows from  the convergence of \(v_k\) to \(v\) in the sense of   \cref{cor:comp} since the functions defined by 
\[
\begin{cases}
\frac{H_k(x)-H_{\alpha_k, \bo{e}_d}}{\alpha_k \eps_k}\qquad &\text{ \(x_d>0\)}\\
\frac{H_k(x)-H_{\alpha_k, \bo{e}_d}}{\beta_k \eps_k}\qquad &\text{ \(x_d<0\)}
\end{cases}
\]
converges (again in the sense of \cref{cor:comp}) to the function 
\[
\bo{v}\cdot x'+p x_d^+-qx_d^-.
\]
\end{proof}

\begin{proof}[\bf Proof of \cref{l:big_regime}]
 Arguing by contradiction one assume for fixed \(\gamma \in (0,1)\) the existence of a sequence of of  functions \(u_k\)  and  numbers \(\alpha_k\), \(M_k\to \infty\)  such that  
\[
\eps_k=\norm{u_k-H_{\alpha_k,\bo{e}_d}}_{L^\infty(B_1)}\to 0 \,,\quad\text{and}\quad   \frac{\alpha_k-\lambda_+}{\eps_k}\ge  M_k\to \infty,
\]
but for which  \eqref{e:iof_normal_oscillation} and \eqref{e:iof_thesis} for any choice of \(\rho\) and \(C\). This implies that \(\ell=\infty\) and that the limiting functions \(v\) obtained in \cref{cor:comp} are solutions of \eqref{e:trasmission}. One then concludes the proof as above by using \eqref{l:reg_trans}.
\end{proof}

\section{Proof of the main results}\label{sec:4}
\subsection{Proof of \cref{thm:main} and \cref{cor:full_free_boundary}}\label{sub:sec:4}
The final step to obtain the desired regularity result is to show that \(|\nabla u^\pm|\) are \(C^\eta\) for a suitable \(\eta>0\)  up to the boundary. This indeed  implies  that $u^\pm$ are solutions of the classical one-phase free boundary problem in its viscosity formulation and the regularity will follows form \cite{desilva}.  The argument is similar to the one in \cite{spve}, therefore we only sketch the main steps and refer the reader to that paper for more details.

\begin{lemma}\label{l:fiorentina_merda}
	Suppose that $u$ is a local minimizer of $ J_{\textsc{tp}}$ in $D$. Then at every point of $\Gamma_{\text{\sc tp}}$ there is a unique  blow-up, that is, 
	\[
	\mathcal BU(x_0)=\{H_{\alpha(x_0), \bo{e}(x_0)}\}.
	\]
	Moreover there exists \(\eta>0\) such that for every open set \(D'\Subset D\) there is a constant $C(D', \lambda_{\pm}, d) >0$ such that, for every $x_0,y_0\in\Gamma_{\text{\sc tp}}\cap D'$, we have
	\begin{equation}\label{e:holder}
	|\alpha(x_0)-\alpha(y_0)|\le C|x_0-y_0|^\eta   
	\qquad\mbox{and}\qquad
	|\bo{e}(x_0)-{\bo{e}}(y_0)|\le C_0|x_0-y_0|^\eta,
	\end{equation}
	where $H_{{\bo{e}}(x_0),\alpha(x_0)}$ and $H_{{\bo{e}}(x_0),\alpha(x_0)}$ are the blow-ups at $x_0$ and $y_0$ respectively.
	In particular, $\Gamma_{\text{\sc tp}}\cap D'$ is locally a closed subset of the graph of a $C^{1,\eta}$ function.
\end{lemma}

\begin{proof}
	We first notice that by \cref{c:flatness} and the definition of $\mathcal{BU}(x_0)$, given $\eps_0>0$ as in  \cref{t:eps_regularity} we can find $r_0>0$ and \(\rho_0\)  such that  \eqref{e:start} is satisfied by $u_{y_0,r_0}$  for some $H_{\alpha, \bo{e}}\in \mathcal{BU}(x_0)$ and for all \(y_0\in B_{\rho_0}(x_0)\).
	
We can thus repeatedly apply \cref{t:eps_regularity}  together with   standard  arguments to infer that for all \(y_0\in B_{\rho_0}(x_0)\) there  exists a unique \(H_{{\bo{e}}(y_0),\alpha(y_0)}\) such that 
	\begin{equation}\label{e:convunif}
	\|u_{r,x_0}-H_{{\bo{e}}(y_0),\alpha(y_0)}\|_{L^\infty(B_r(y_0))}\le C_0 r^\gamma
	\end{equation}
	where  $\gamma\in (0,\sfrac12)$. A covering argument implies the validity of the above estimate for all \(x_0\in \Gamma_{\text{\sc tp}}\cap D'\). Next, for  \(x_0, y_0\in \Gamma_{\text{\sc tp}}\cap D'\) set $r:=|x_0-y_0|^{1-\eta}$ and $\eta:=\sfrac{\gamma}{(1+\gamma)}$, and recall that $u$ is $L$-Lipschitz  (with constant depending on \(D'\)) to get
	\begin{align*}
	\|&H_{{\bo{e}}(x_0),\alpha(x_0)}-H_{{\bo{e}}(y_0),\alpha(y_0)}\|_{L^\infty(B_1)}\\
	&\leq \|u_{r,x_0}-H_{{\bo{e}}(x_0),\alpha(x_0)}\|_{L^\infty( B_1)}+\|u_{r,x_0}-u_{r,y_0}\|_{L^\infty( B_1)}+\|u_{r,y_0}-H_{{\bo{e}}(y_0),\alpha(y_0)}\|_{L^\infty( B_1)}\\
	&\leq \left( C_0 r^{\gamma} + \frac{L}r{\left|x_0-y_0\right|}+C_0 r^{\gamma}  \right)=(L+2C_0)\,  \left|x_0-y_0\right|^{\eta}\,.
	\end{align*}
	The conclusion now follows easily from this inequality.
\end{proof}

\begin{lemma}\label{l:visc}
	Under the same assumptions of \cref{l:fiorentina_merda}, there are $C^{0,\eta}$ continuous functions $\alpha \colon \partial\Omega_u^+\to\R$,   $\beta \colon \partial\Omega_u^-\to\R$  such that $\alpha\ge \lambda_+$ , \(\beta \ge \lambda-\),   and $u^\pm$ are viscosity solutions of the one-phase problem
	\begin{equation*}
	\Delta u^+=0\quad\text{in}\quad\Omega_u^+\,,\qquad |\nabla u^+|=\alpha \quad\text{on}\quad \partial\Omega_u^+\,.
	\end{equation*}
and 
	\begin{equation*}
	\Delta u^-=0\quad\text{in}\quad\Omega_u^-\,,\qquad |\nabla u^-|=\beta \quad\text{on}\quad \partial\Omega_u^-\,.
	\end{equation*}
\end{lemma}

\begin{proof}	We will  sketch the argument for $u^+$, $u^-$ being the same. Clearly $\Delta u^+=0$ in $\Omega_u^+$. By \eqref{e:convunif} we have that, if $x_0\in \Gamma_{\text{\sc tp}}\cap D'$, then 
		\begin{equation}\label{e:pestorosso}
		\big|u^+(x)-\alpha(x_0)(x-x_0)\cdot \bo{e}(x_0)\big|\le C_0|x-x_0|^{1+\gamma}\quad \text{for every}\quad x\in B_{r_0}(x_0)\cap \Omega_u^+\,,
		\end{equation}
		where \(r_0\) and \(C_0\) depends only on \(D'\).
		In particular, $u^+$ is differentiable on $\Omega_u^+$ up to $x_0$ and $|\nabla u^+(x_0)|=\alpha (x_0)$. On the other hand if $x_0\in \Gamma_{\textsc{op}}^+:=\Omega_u^+\setminus \partial \Omega_u^-$, then $|\nabla u^+(x_0)|=\lambda_+$ is constant, in the viscosity sense. 
		
		To conclude we only need to prove that $\alpha \in C^{0,\eta }(\de \Omega_+)$. Since $\alpha $ is \(\eta\)  H\"older continuous on $\Gamma_{\text{\sc tp}}$ by \cref{l:fiorentina_merda} and constant on $\Gamma_{\textsc{op}}^+$, we just need to show that if $x_0\in\Gamma_{\text{\sc tp}}$ is such that there is a sequence  $x_k\in \Gamma_{\textsc{op}}^+$ converging to $x_0$, then $\alpha(x_0)=\lambda_+$. To  this end, let $y_k \in \Gamma_{\text{\sc tp}}$ be such that 
		\[
		\dist(x_k, \Gamma_{\text{\sc tp}})=|x_k-y_k|\,.
		\]
		Let us set 
	$$
	r_k=|x_k-y_k|\qquad\text{and}\qquad u_k(x)=\frac1{r_k}u^+(x_k+r_kx),
	$$ 
	 and note that $u_k$ is a viscosity solution of the free boundary problem 
	$$
	\Delta u_k=0\quad\text{in}\quad\Omega_{u_k}^+\cap B_1\,,\qquad |\nabla u_k|=\lambda_+\quad\text{on}\quad \partial\{u_k>0\}\cap B_1\,.
	$$
	Since $u_k$ are uniformly Lipschitz they converge to a function $u_\infty$ which is also a viscosity solution of the same problem,  \cite{desilva}. On the other hand, by \eqref{e:pestorosso}, we have that
	 \[
	 u_\infty(x)=\alpha(x_0)(x\cdot \bo{e}(x_0))^+,
	 \] 
	 which gives that \(\alpha(x_0)=\lambda_+\).
\end{proof}

\begin{proof}[Proof of \cref{thm:main}] 
Let  $x_0\in\Gamma_{\text{\sc tp}}=\de \Omega_u^+\cap \de \Omega_u^-$ and let $\bar{\eps}$ be the constant in \cite[Theorem 1.1]{desilva}. Thanks to the classification of blow-ups at points of $\Gamma_{\text{\sc tp}}$, we can choose $r_0>0$, depending on $x_0$, such that 
\[
\|u_{x_0,r_0}-H_{\alpha,\bo{e}} \|_{L^{\infty}(B_1)}<\bar{\eps}
\]
so that thanks to \cref{l:visc}, we can apply \cite[Theorem 1.1]{desilva} to conclude that locally at $x_0\in \Gamma_{\text{\sc tp}}$ the free boundaries $\de \Omega_u^\pm$ are $C^{1,\eta}$ graphs. By the arbitrariness of $x_0$ this concludes the proof.
\end{proof}	
	
\begin{proof}[Proof of \cref{cor:full_free_boundary}]
The proof of the corollary is straightforward. Indeed by \cref{thm:main} there  exits an open neighborhood \(W\) of the two-phase free boundary \(\Gamma_{\textsc{tp}}\) such that  \(\partial \Omega_u^\pm \cap W\subset \Reg(\partial \Omega_u^\pm) \). Outside \(W\), \(u^\pm\) are (local) minimizers of the one-phase problem and thus the desired decomposition and the stated properties follows by the results in \cite{altcaf,EdEn,weiss}.
\end{proof}

\subsection{Proof  of \cref{t:multiphase}}\label{sub:intro:multi}
In this section we prove the regularity of the solutions to the shape optimization problem \eqref{e:multi}. The proof is a consequence of \cref{thm:main} and the analysis in \cite{stv}. 
Indeed, the existence of an optimal (open) partition $(\Omega_1,\dots,\Omega_n)$ was proved in \cite{bucve} and (in dimension two) in \cite{benve}. 
Moreover, in \cite{bucve} and \cite{mono}, it  has been  shown that each of the eigenfunctions $u_i$ on $\Omega_i$ is Lipschitz continuous 
as a function defined on $\R^d$ (extended as zero outside $\Omega_i$). Furthermore,   there are no triple points inside the box $D$ and no two-phase points on the boundary $\partial D$, that is,
\begin{itemize}
\item $\partial\Omega_i\cap\partial\Omega_j\cap\partial\Omega_k=\emptyset$ for every set of different coefficients $\{i,j,k\}\subset\{1,\dots,n\}$;
\item $\partial\Omega_i\cap\partial\Omega_j\cap\partial D=\emptyset $ for every $i\neq j\in\{1,\dots,n\}$.
\end{itemize}
The regularity of   $\partial\Omega_i$ can then be obtained as follows.
\begin{itemize}
\item By \cite[Lemma 7.3]{stv}, the function \(u=u_i-u_j\) is a almost of \eqref{e:J1} with \(\lambda^2_+=m_i\) and \(\lambda^2_-=m_j\), in the sense that 
\[
 J_{\textsc{tp}}(u, B_r)\le J_{\textsc{tp}}(v, B_r)+Cr^{d+2}\qquad \text{for all \(v=u\) on \(\partial B_r\)}\,,
\]
provided \(r\) is sufficiently small.
\item By the classification of the blow up limits in \cite[Proposition 4.3]{stv} and the arguments in \cref{sub:viscosity}, \(u\) is a viscosity solution of 
\begin{equation*}
\begin{cases}
\Delta u= -\lambda_1(\Omega_i)u_i+\lambda_1(\Omega_j)u_j\qquad  &\text{on } \{u\ne 0\}
\\
 |\nabla u^+|^2-|\nabla u^-|^2=m_i-m_j ,\,
|\nabla u^+|\ge \sqrt{m_i}\quad\text{and}\quad |\nabla u^-|\ge \sqrt{m_j} &\text{on } \partial \Omega_u^+\cap \partial \Omega_u^-\,;
\\
  |\nabla u^+|=\sqrt{m_i}&\text{on }\partial  \Omega_u^+\setminus \partial \Omega_u^-\,; \\
   |\nabla u^-|=\sqrt{m_j}&\text{on }  \partial  \Omega_u^-\setminus \partial \Omega_u^+\,.
\end{cases}
\end{equation*}

\item  $C^{\infty}$ regularity of the one-phase part $\partial\Omega_i\setminus \bigl(\partial D\cup \bigl(\bigcup_{i\neq j}\partial\Omega_j\bigr)\bigr)$ follows by techniques in  \cite{altcaf}, see \cite{brla};
\item  $C^{1,\eta}$-regularity of $\partial\Omega_i$ in a neighborhood of $\partial\Omega_i\cap\partial D$ was proved in \cite{rtv}; the main argument boils down to the regularity result from \cite{changsavin};
\item  $C^{1,\eta}$-regularity of $\partial\Omega_i$ in a neighborhood of $\partial\Omega_i\cap\partial \Omega_j$ follows by  using the same arguments\footnote{Note that  \(\Delta u_r(x)=r\Delta u (rx)\).  Hence, since \(\Delta u\) is uniformly bounded in \(L^\infty\), \(\|\Delta u_r\|_{L^\infty}=O(r)\)  and thus  this does not  interfere  with the iteration argument.} in the proof of \cref{thm:main},  using \cref{thm:below} in place of  \cref{t:eps_regularity}.
\end{itemize} 

\begin{theorem}\label{thm:below}
Let \(0\le \lambda_+\le \lambda_-\le L\), \(f \in C^0(B_1)\)  and let  \(u:B_1 \to \R\) be a  \(L\)-Lipschitz viscosity solution of 
\begin{equation*}
\begin{cases}
\Delta u= f\qquad  &\text{on }\{u\ne 0\}
\\
 |\nabla u^+|^2-|\nabla u^-|^2=\lambda_+^2-\lambda_-^2 ,\,
|\nabla u^+|\ge \lambda_+\quad\text{and}\quad |\nabla u^-|\ge \lambda_- &\text{on }\partial \Omega_u^+\cap\partial \Omega_u^-\,;\\
  |\nabla u^+|=\lambda_+&\text{on } \partial \Omega_u^+\setminus\partial\Omega_u^-\,; \\
   |\nabla u^-|=\lambda_-&\text{on }  \partial \Omega_u^-\setminus \partial \Omega_u^+\,.
\end{cases}
\end{equation*}
Then for every  $\gamma\in(0,\sfrac12)$, there exist $\eps_0>0$, $C>0$ and $\rho\in(0,\sfrac14)$ depending only on \(\lambda_{\pm}\), \(L\) and \(\gamma\) such that if 
\begin{equation*}
\|u-H_{\alpha,\bo{e}_d}\|_{L^\infty(B_1)}\le \eps_0\qquad\text{for some}\qquad L\ge \alpha\ge\lambda_+\,.
\end{equation*}
then, there are $\bo{e}\in\mathbb S^{d-1}$ and $\tilde\alpha\ge \lambda_+$ such that 
\begin{equation}\label{e:iof_normal_oscillation2}
|\bo{e}-\bo{e}_d|+|\tilde\alpha-\alpha|\le C\bigl(\|u-H_{\alpha,\bo{e}_d}\|_{L^\infty(B_1)}+\|f\|_{L^\infty(B_1)}\bigr)
\end{equation}
and
\begin{equation*}\label{e:iof_thesis2}
\|u_\rho-H_{\tilde\alpha,\bo{e}}\|_{L^\infty(B_1)}\le \rho^{\gamma}\,\|u-H_{\alpha,\bo{e}_d}\|_{L^\infty(B_1)}+C\|f\|_{L^\infty(B_1)}
\end{equation*} 

\end{theorem}

\begin{proof}
Note that \eqref{e:iof_thesis2}  is satisfied with \(\tilde \alpha=\alpha\) and \(\bo{e}=\bo{e}_d\), \(\rho=\sfrac{1}{4}\) and \(C=C(\eps)\) if
\[
\|f\|_{L^\infty (B_1)}\ge \eps \|u-H_{\alpha,\bo{e}_d}\|_{L^\infty(B_1)}.
\]
Hence it is enough to show that there exists \(\eps_0\) universal such that the conclusion of the theorem holds provided 
\begin{equation*}
\|u-H_{\alpha,\bo{e}_d}\|_{L^\infty(B_1)}\le \eps_0\qquad\text{for some}\qquad L\ge \alpha\ge\lambda_+\,
\end{equation*}
and 
\begin{equation}\label{e.opicolo}
\|f\|_{L^\infty (B_1)}\le \eps_0 \|u-H_{\alpha,\bo{e}_d}\|_{L^\infty(B_1)}.
\end{equation}
We can then argue by contradiction as in the proof of \cref{t:eps_regularity} by noticing that, thanks to \eqref{e.opicolo} the contradicting sequence satisfies 
\[
\Delta u_k=o(\eps_k).
\]
This allows to almost verbatim repeat the proofs in \cref{sec:3}, see for instance \cite{desilva,dfs}.
\end{proof}

\appendix
\section{Proof of \cref{l:touching}} \label{s:app}
Here we prove \cref{l:touching}. The idea to construct the comparison functions is to perform (the inverse of)  the   changed of variable used in  \cite{KN} (and attributed to Friederichs) which maps, for smooth solutions,  the free boundary problem, to a fixed boundary (non linear) problem on a fixed domain, see \cite[Section 3]{KN}.

\begin{proof}[Proof of \cref{l:touching}]
We divide the proof into several steps:

\medskip
\noindent
\(\bullet\)\emph{Step 1}: Given  \(\alpha>0\) function \(P\in  C^1(\overline{B^+_{\sfrac{1}{2}}})\cap C^2(\overline{B^+_{\sfrac{1}{2}}}) \) there exist \(\bar \eps \ll1 \), depending only on the \(C^1\) norm of \(P\) such that for all \(\eps\le \bar \eps\) there exists a  function   \(Q\in C^1 (\overline{\{Q>0\}})\cap C^2(\{Q>0\})\) such that
\begin{equation}\label{e:juvemerda}
Q_\eps(y', y_d-\eps \alpha  P(y',y_d)))=\alpha y_d\qquad \text{for all \(y=(y',y_d)\in \overline{\{Q>0\}}\)}
\end{equation}
To this end we define the following map \(\bo{T}_\eps:\overline{B^+_{\sfrac{1}{2}}}\to \R^d\):
\[
\bo{T}_\eps(x',x_d)=(x', x_d-\eps \alpha P(x',x_d))\qquad x=(x',x_d)\in B^+_{\sfrac{1}{2}}.
\]
Note that if \(\eps \ll \|P\|_{C^1}^{-1}\) , \(\bo{T_\eps}\) induces a bijection between \(B^+_{\sfrac{1}{2}}\) and \(U_\eps:=\bo{T}_\eps(B^+_{\sfrac{1}{2}})\subset B_1\). We let \(\bo{Q}_\eps\) be its inverse and we define \(Q_\eps\) as its \(d\)-th component times \(\alpha\), namely
\[
Q_\eps:=\alpha (\bo{Q}_\eps\cdot \bo{e}_d ): U\to (0,1/2),
\]
and we extend it to zero on \(B_{\sfrac1 2} \setminus \overline{\{Q>0\}}\).
It is now immediate to verify that \eqref{e:juvemerda} is satisfied. Furthermore, with  the notation \(y_\eps=\bo{T}\eps(x)\),
\begin{equation}\label{e:juvemerda2}
\nabla Q_\eps(y_\eps)=\alpha \bo{e_d}+\alpha\eps \nabla P(x)+O(\eps^2)
\end{equation}
and
\begin{equation}\label{e:juvemerda3}
\Delta Q(y_\eps)=\alpha^2 \eps \Delta P (x)+O(\eps^2)
\end{equation}

\medskip
\noindent
\(\bullet\)\emph{Step 2}: Let us now prove item  (i) of the statement, item  (ii) can be obtained by a symmetric argument.  Let \(\alpha_k\), \(\eps_k\) a be as in the statement. Let us assume that \(P_+\) is a strictly subharmonic function touching \(v_+\) strictly form below at \(x_0\). By assumption, for all \(\delta\ll 1\) the function \(v_+-P_-+\delta\) has  a strictly positive minimum at  \(  x_0\) as \(\delta \to 0\). Let \(Q^\delta_k\) be the functions constructed in Step 1 with \(\eps=\eps_k\),  \(\alpha=\alpha_k\) and \(P=P_--\delta\). Let us define 
\[
P^\delta _k(x)=\frac{Q^\delta_k-\alpha_k x_d^+}{\alpha_k \eps_k}
\]
and 
\[
\tilde \Gamma_k=\bigl\{ (x,P^\delta _k(x)) \qquad x \in \overline{\{Q^\delta_k>0\} \cap B_{\sfrac{1}{2}}} \bigr\}.
\]
One easily checks that they converge in the Hausdorff distance to 
\[
\tilde \Gamma = \bigl\{ (x,P_+(x)-\delta) \qquad x \in \overline{B^+_{\sfrac{1}{2}} }\bigr\}.
\]
By using that the graphs \(\Gamma_k^+\) defined in \cref{cor:comp} converges in the Hausdorff distance to 
\[
\tilde \Gamma = \bigl\{ (x, v_+ (x)) \qquad x \in \overline{B^+_{\sfrac{1}{2}}}\bigr\}.
\]
We claim that 
\[
\{Q^\delta_k>0\}\cap B_{\sfrac{1}{2}}\Subset  \{u_k>0\}\cap B_{\sfrac{1}{2}}.
\]
Indeed otherwise one would find a sequence of points \(x_k\) such that \(Q^\delta_k(x_k)>0\) and \(u^+(x_k)=0\) which implies that  
\[
P^\delta _k(x_k)\ge v_{+, k}(x_k)
\]
where \(v_{+, k}\) is define in \cref{cor:comp}. Assuming that \(x_k\to \bar x\)  we get \(P^+(\bar x)-\delta \ge v_+(\bar x) \) in contradiction with   \(P-\delta< v_+\).

In particular there exists \(\sigma=O(\delta)\) such that \(Q^\delta_k(\cdot-\sigma \bo{e}_d)\) touches \(u^+\) from below at some point \(x^\delta_k\). Note also that, arguing as above)  \(x^\delta_k\to x_0\) as \(k\) goes to infinity. By \eqref{e:juvemerda3}  and the strict subharmonicity of \(P\) one has that 
\[
\Delta Q^\delta_k>0 \qquad \text{on \(Q^\delta_k>0\)}
\]
Hence the touching point lies on the free boundary \(\partial \Omega_{u_k}^+\). Furthermore by \eqref{e:juvemerda2}
\[
\begin{split}
\nabla Q^\delta _\eps(x_k^\delta)&=\alpha \bo{e_d}+\alpha\eps \nabla P_+(\bo{Q}_{\eps_k}(x^k_\delta))+O(\eps^2)
\\
&=\alpha \bo{e_d}+\alpha\eps \nabla P_+(x_0)+\eps_kO(\abs{x^k_\delta-x_0})+O(\eps^2)
\end{split}
\]
Choosing a sequence \(\delta_k\to 0\) we obtain the desired conclusion.

\medskip
\noindent
\(\bullet\)\emph{Step 3}: We now prove item (iii). The proof goes exactly as above, more precisely we let \(P\) be as in the statement and we define \(P^\pm \) as \(P\) restricted to \(B_{\sfrac{1}{2}}^\pm\). We let also \(\bo{T}^\pm: B_{\sfrac{1}{2}}^\pm\) be the corresponding transformations as in Step 1(with \(\bo{T}^-\) defined in the obvious way on \(B_{\sfrac{1}{2}}^-\)). The key point is to note that 
\[
\bo{T}^+(B_{\sfrac{1}{2}}^+)\cap \bo{T}^-(B_{\sfrac{1}{2}}^-)=\emptyset.
\]
Hence, with obvious notation, the function \footnote{Note that if \(Q^-\) is the \(d\)-th component of the inverse of \(\bo{T}^-\) then it is \emph{negative}!}
\[
Q=Q^++Q^-
\]
is a well defined comparison function. Arguing as in Step 2 gives the desired sequence.
\end{proof}

\section{Proof of \cref{l:reg_twom}}\label{app:twomembrane}

Give a solution \(v\) we  define \(w\)
\[
w_\pm(x)=v_{\pm}(x)-\frac{\ell}{\lambda_\pm^2} x_d\qquad x \in B^{\pm}_{\sfrac{1}2}.
\]
It is straightforward to check it is a  viscosity solution of 

\begin{equation*}
		\begin{cases}
		\Delta w_{\pm}=0 \qquad&\text{in \(B^\pm_{\sfrac{1}{2}}\)}
		\\
	\partial_{d} w_{\pm} \ge 0 &\text{in \(B_{\sfrac{1}{2}}\cap\{x_d=0\}\)}
		\\
	\partial_{d} w_{\pm}  =0 &\text{in \(\mathcal J\)}
		\\
		\lambda_+^2\partial_{d} w_{+} =\lambda_-^2 \partial_{d} w_{-}  &\text{in \(\mathcal C\)}
		\\
		 w_+\leq w_- & \text{in \(B_{\sfrac{1}{2}}\cap \{x_d=0\}\) }
		\end{cases}\,.
		\end{equation*}
Furthermore one  can easily check that  
\[
w_\pm(x',x_d)=\frac{1}{\lambda_\pm^2} w_N(x', \mp x_d)-w_S(x',\mp x_d)\,,
\]
where \(w_N\) solves the Neumann problem 
\begin{equation*}\label{e:lin1}
\begin{cases}
\Delta w_N=0 & \text{on }B_{\sfrac12}^-
\\
\de_dw_N=0&\text{on }B_{\sfrac12}^-\cap\{x_d=0\}
\end{cases}\,,
\end{equation*}
and  \(w_S\) is a solution of the thin obstacle problem
\begin{equation*}\label{e:lin2}
\begin{cases}
\Delta w_S=0 & \text{on }B_{\sfrac12^-}
\\
w_S\geq 0 &\text{on }B_{\sfrac12}^-\cap\{x_d=0\}
\\
\de_dw_S\geq0&\text{on }B_{\sfrac12}^-\cap\{x_d=0\}
\\
w_S\,\de_dw_S=0&\text{on }B_{\sfrac12}^-\cap\{x_d=0\}
\end{cases}\,.
\end{equation*}
Clearly \(w_N\in C^{\infty}(\overline{B_{\sfrac14}^+})$  with 
\[
\|w_N\|_{C^{k}(B_{\sfrac14})}\leq C_k\|w_N\|_{L^\infty(B_{\sfrac{1}2})}.
\]
On the other hand, by \cite[Corollary pg. 58]{AtCa}, $w_S\in C^{1,\sfrac12}(\overline{B_{\sfrac14}^+})$ with 
\[
\|w_S\|_{C^{1,1/2}(B_{\sfrac14})}\leq C\|w_S\|_{L^\infty(B_{\sfrac{1}2})}.
\]
From the last two estimates and the definition of \(w\) it is easy to deduce the conclusion of the Lemma. 
\qed

\end{document}